\documentclass[11pt]{amsart}
	\usepackage{amssymb, latexsym, amsmath}
	\usepackage[pdftex]{graphicx}
	\usepackage{float}
	\usepackage[numbers]{natbib}
	\usepackage{hyperref}
	\usepackage{tikz-cd}
	\newtheorem{thm}{Theorem}[section]
	
	\newtheorem{lem}[thm]{Lemma}
	\newtheorem{prop}[thm]{Proposition}
	\theoremstyle{definition}
	\newtheorem{defn}[thm]{Definition}
	\theoremstyle{remark}
	\newtheorem{rem}[thm]{Remark}
	\numberwithin{equation}{section}
	\newtheorem{ex}[thm]{Example}
	\numberwithin{equation}{section}

	\newcommand{\mbb}{\mathbb}

	\newcommand{\ov}{\overline}
	
	\newcommand{\ep}{\epsilon}
	\newcommand{\no}{\noindent}

	\newcommand{\cal}{\mathcal}
	
	\newcommand{\la}{\lambda}

	\newcommand{\cord}{(z_1,z_2, \hdots ,z_k)}
	
\begin{document}
	\title{Polynomial shift--like maps in $\mbb C^k$}
	\author{Sayani Bera} 
	\address{Sayani Bera: School of Mathematics, Ramakrishna Mission Vivekananda Educational and Research Institute, 711202, West Bengal, India}
	\email{sayanibera2016@gmail.com}
	\pagestyle{plain}
	\begin{abstract}
	The purpose of this article is to explore a few properties of polynomial shift-like automorphisms of $\mbb C^k.$ We first prove that a $\nu-$shift-like polynomial map (say $S_a$) degenerates essentially to a polynomial map in $\nu-$dimensions as $a \to 0.$ Secondly, we show that a shift-like map obtained by perturbing a hyperbolic polynomial (i.e., $S_a$, where $|a|$ is sufficiently small) has finitely many Fatou components, consisting of basins of attraction of periodic points and the component at infinity.
	\end{abstract}
	\maketitle
 \section{Introduction}
%
\no Bedford and Pambuccian in \cite{BP}, introduced a class of polynomial automorphisms called the shift--like maps in $\mbb C^k$, $k \ge 3$ which can be viewed as a generalization of H\'{e}non maps to the higher dimensions. They are defined as follows:

\medskip\no 
A polynomial \textit{shift--like map} of type $1 \le \nu \le k-1$ in $\mbb C^k$, $k \ge 3$ is an automorphism of $\mbb C^k$ of the form:
\[ S_a(z_1,z_2, \hdots,z_k)=(z_2,\hdots,z_k,p(z_{k-\nu+1})+a z_1)\]
where $p$ is polynomial (or entire) in one variable and $a \in \mbb C^*$. Henceforth, we will use the phrase \textit{$\nu-$shift of the polynomial $p$} in $\mbb C^k$ to refer such maps.

\medskip\no 
Furthermore, it is noted in \cite{BP} that polynomial shift--like maps share many properties similar to H\'{e}non maps in $\mbb C^2$, such as the existence of a filtration, the construction of Green functions and the corresponding stable and unstable currents. However, unlike H\'{e}non maps, shift--like maps are not \textit{regular}, i.e., the indeterminacy sets of $S_a$ (say ${I}^+$) and $S_a^{-1}$ (say ${I}^-$) are not disjoint in $\mbb P^k.$ Our goal in this article is to extend further, the theory of dynamics of shift--like polynomial automorphisms, in analogy with H\'{e}non maps.

\medskip\no Let us first recall a few standard notations. Let $K^\pm$, $J^\pm$ and $J$ denote the following sets corresponding to an automorphism $F$ of $\mbb C^k$, $k\ge 2$
\[K^\pm=\{ z \in \mbb C^k: F ^{\pm n}(z) \text{ is bounded }\}, J^\pm=\partial K^\pm \text{ and } J=J^+ \cap J^-.\]
For a shift--like map $S_a$, the aforementioned sets will be denoted by $K_a^\pm$, $J_a^\pm$ and $J_a.$ Let $G_a^\pm$
denote the Green function associated to the set $K_a^\pm$ respectively. Further, let $\mu_a^\pm$ denote the stable and unstable current associated $S_a$ which are defined as (see \cite{BP}):
\[ \mu_a^+=\Big(dd^c \frac{G_a^+}{2 \pi}\Big)^\nu \text{ and } \mu_a^-=\Big(dd^c \frac{G_a^-}{2 \pi}\Big)^{k-\nu}\]
where $\nu-$is the the type of the shift--map. Then the wedge product $\mu_a=\mu_a^+ \wedge \mu_a^-$ gives a Borel measure in $\mbb C^k$ such that
\[ {S_a}_*^{\nu(k-\nu)}(\mu_a)=\mu_a.\]
We will first investigate the behavior of the measure $\mu_a$ as $S_a$ degenerates to a $\nu-$dimensional map with $a \to 0.$ The degeneration in the case of H\'{e}non maps was studied by Bedford--Smillie in \cite{BS3}. 

\medskip\no Next, we are interested in studying whether the action of a shift--like map is hyperbolic on $J_a$ or not. Let us first recall the definition of \textit{hyperbolicity}.

\medskip\no A smooth diffeomorphism $F$ of a manifold $M$ equipped with a Riemannian  norm $|\cdot|$ is said to be {\it hyperbolic} (see \cite{KatokBook}) on a compact completely invariant subset $S \subset M$ if 
there exist constants $\la >1$, $C>0$ and a continuous splitting of the of the tangent space $T_x X$ for every $x \in S$ into $E_x^s \oplus E_x^u=T_x X$ such that:
\begin{itemize}
\item[(i)] $DF(x)(E_x^s)=E^s_{F(x)}$ and $DF(x)(E_x^u)=E^u_{F(x)}$.

\medskip
\item[(ii)] $|DF^n(x)v| \le C\la ^{-n}|v|$ for $v \in E_x^s$ and $|DF^n(x)v| \ge C^{-1}\la^n |v|$ for $v \in E_x^u.$
\end{itemize}
\medskip\no It is a well known fact from \cite{Hu-OV} and \cite{BS1}, that if a H\'{e}non map (say $H_a$) is obtained by a small enough perturbation of a \textit{hyperbolic polynomial} in one variable then $H_a$ is hyperbolic on $J_a.$ Furthermore, the Fatou set of a hyperbolic H\'{e}non map in $\mbb C^2$ consists of finitely many sinks and the component at infinity(see \cite{BS1}).

\medskip\no In this article, we will prove that if a \textit{1-shift--like map} $S_a$ in $\mbb C^k$ is obtained by a sufficiently small perturbation of a hyperbolic polynomial then $S_a$ is hyperbolic on $J_a$ with respect to a the Riemannian metric (equivalent to the Euclidean metric). As a consequence, we obtain that the interior of $K_a^+$ is union of finitely many sinks. However, for any arbitrary $\nu-$shift ($1 \le \nu \le k-1$), we are not able to conclude whether $S_a$ is hyperbolic on $J_a$ or not, but we prove that a sufficiently small $\nu-$shift of a hyperbolic polynomial has finitely many Fatou components and the interior of $K_a^+$ consists of finitely many sinks. This paper is organized as follows: 

\medskip\no 
In Section \textbf{2}, we recall some basic properties of shift--like maps from \cite{BP}. Further, we prove two facts. First, the $\nu(k-\nu)-$th iterate of a $\nu-$shift of a polynomial of degree $d \ge 2$ is regular and hence we can work with appropriate compositions of shift maps. Second, if $ |a|<1$ then for any $\nu-$shift $S_a$ of a polynomial, the set $K_a^-$ has empty interior, i.e., $K_a^-=J_a^-.$ The proof of this is similar to that for H\'{e}non maps.
   
\medskip\no 
In Section \textbf{3}, we investigate the behavior of $S_a$ as it degenerates with $a \to 0.$ By \textit{degeneration} we mean $$S_a(z_1,\hdots,z_k) \to S_0(z_1,\hdots,z_k)=(z_2,\hdots,z_k,p(z_{k-\nu+1}))$$
as $a \to 0.$ We show that for a $\nu-$shift there exists an appropriate graph $\Gamma_{\nu} $ in $\mbb C^k$ over $\mbb C^\nu$, i.e., a biholomorphism $\phi: \mbb C^\nu \to  \Gamma_{\nu}$ such that $S_0(\mbb C^k)=\Gamma_\nu.$ For $\nu \ge 1$, let $p_{\nu}: \mbb C^{\nu} \to \mbb C^{\nu} $ denote the map
\[ p_{\nu}(z_1,\dots,z_{\nu})=(p(z_1),\hdots,p(z_{\nu}))\]
where $p$ is a monic polynomial map of degree $d \ge 2$ in $\mbb C.$
The Green function and the Borel measure $\mu_{p_\nu}$ corresponding to the map $p_{\nu}$ are defined as:
\[ G_{p_{\nu}}(z)=\lim_{n \to \infty}\frac{\log^+\|p_\nu^n(z)\|}{d^n} \text{ and }\mu_{p_\nu}=\Big(dd^c \frac{G_{p_\nu}}{2\pi}\Big)^{\nu}.\] 
\begin{thm}\label{degeneration}
Let $p$ be a monic polynomial and $\mu_p$ be the equilibrium measure on $\mbb C$, supported in $K_p$ corresponding to the polynomial $p.$ Then for a $\nu-$shift of the polynomial $p$
\[ \lim_{a \to 0} \mu_a=\phi_*(\mu_{p_\nu}) \text{ and }\limsup_{a \to 0} \Lambda_a=(k-\nu)\Lambda_{p_\nu}.\]
where $\Lambda_{p_\nu}$ and $\Lambda_a$ are the Lyapunov exponent of the map $p_{\nu}$ and $S_a^{\nu(k-\nu)}$ with respect to the measure $\mu_{p_\nu}$ in $\mbb C^\nu$ and $\mu_a$ in $\mbb C^k.$
\end{thm}
\no  The idea of the proof is that as $a \to 0$, a $\nu-$shift degenerates essentially to the map $p_\nu$ on the graph $\Gamma_\nu$ and the support of the $(k-\nu,k-\nu)-$current $\mu_a^-$ converges to $\Gamma_\nu.$ This result is an extension of Proposition 6.3 from \cite{BS3} for shift-like polynomials.

\medskip\no 
In Section \textbf{4}, we recall a few properties of \textit{hyperbolic} polynomials in $\mbb C$ and a prove a fact regarding the perturbed dynamics of the polynomial in a closed subset contained in $F_\infty(p)$ (i.e., the Fatou component of $p$ containing infinity). 

\medskip\no In Section \textbf{5}, we prove the following two results in $\mbb C^3.$ The generalized results for $\mbb C^k$ are stated as:
\begin{thm}\label{main theorem 1}
Let $S_a$ be a polynomial shift--like map of type $1$ in $\mbb C^k$ of the form, i.e.,
\[S_a\cord=(z_2,z_3,\hdots,az_1+p(z_k))\] where $p$ is a hyperbolic polynomial in $\mbb C$ with connected Julia set.
Then there exists $A>0$ such that for every $0<|a|<A$, $S_a^{k-1}$ is hyperbolic on $J_a$ with respect to a Riemannian metric (equivalent to the Euclidean metric) on a neighbourhood of $J_a.$
\end{thm}
\begin{thm}\label{main theorem 2}
Let $S_a$ be a polynomial shift--like map of type $1 \le \nu \le k-1$ in $\mbb C^k$ of the form
\[S_a\cord=(z_2,z_3,\hdots,az_1+p(z_{k-\nu+1}))\] where $p$ is a hyperbolic polynomial in $\mbb C$ with connected Julia set. Then there exists $A>0$ such that for every $0<|a|<A$, $S_a$ satisfies the following properties:
\begin{itemize}
\item[(a)]There is no wandering domain of $S_a$.
\item[(b)]Each component in the interior of $K_a^+$ is the basin of attraction of an attracting
periodic point.
\item[(c)]There are at most finitely many basins of attraction.
\end{itemize} 
\end{thm}
\no The proof of Theorem \ref{main theorem 1} is similar to the proof known for H\'{e}non maps in \cite{BS3} and uses the existence of invariant cone fields. Also Theorem \ref{main theorem 2}, for type $1-$shifts is a direct consequence of the fact that $S_a^{k-1}$ is regular and hyperbolic (see \cite{safikov-wolf}).

\medskip\no To prove Theorem \ref{main theorem 2} for type $2-$shifts in $\mbb C^3$, we first show that there exist appropriate closed invariant (under $S_a^2$) subsets $J_i$, $1 \le i \le 3$ such that $J_i \subset J_a$ and for each $i$, $S_a^2$ is hyperbolic on each of this $J_i$ in an appropriate neighbourhood.  Finally, we conclude the result from the fact that $J_a$ is  maximal in a certain neighbourhood and the stable manifold of $J_a$ is equal to the stable manifold of union of $J_i$, i.e., $W^s(J_a)=W^s(J_1 \cup J_2 \cup J_3).$ 

\medskip\no In Section \textbf{6}, we give a sketch of the proof of Theorem \ref{main theorem 1} and \ref{main theorem 2} in $\mbb C^k$ for any $\nu-$shift. The ideas are similar to the case $k=3.$ 

\medskip\no 
\textbf{Acknowledgements:} The author would like to thank Sushil Gorai and Kaushal Verma for helpful discussions and suggestions regarding the problem

\section{Some properties of shift--like maps}
\no In this section, we will first recall a few properties of shift--like maps from \cite{BP}. 

\medskip\no 
For $R>0$, let $V_R=\{z \in \mbb C^k: |z_i| \le R,\; 1 \le i \le k\}.$ Also for every $i$, $1 \le i \le k$
\[ V_R^i=\{z \in \mbb C^k: |z_i| \ge \max(|z_j|,R),\; 1 \le j \neq i \le k\}.\]
For a $\nu-$shift, $S_a$ the sets $V_R^+$ and $V_R^-$ is defined as:
\[ V_{R}^+=\bigcup_{i=k-\nu+1}^k V^i_R \text{ and } V_{R}^{-}=\bigcup_{i=1}^{k-\nu} V^i_R\] for every $R>0$.

\medskip\no 
Note that $\mbb C^k=V_R \cup V_R^+ \cup V_R^-$ and there exists $R>0$ (sufficiently large) such that $V_{R}$, $V_{R}^+$ and $V_{R}^-$ gives a  {\it filtration} of $\mbb C^k$ for the map $S_a$ (see \cite{BP}).

\medskip\no A few properties of these sets are stated below:
\begin{enumerate}
\item For a fixed $a \in \mbb C^*$, there exists $R_a \gg 1$ such that for every $R \ge R_a$
\[ S_a(V_{R}^+) \subset V_{R}^+.\] Also, from this it follows that for every $n \ge 0$
\[ S_a^{ -n}(V_{R}^+) \subset S_a^{ -(n+1)}(V_{R}^+).\]
\item Let $U_a^+=\{z \in \mbb C^k: S_a^{ n}(z) \to \infty \text{ as } n \to \infty\}$ and $K_a^+=\{ z \in \mbb C^k: S_a ^{n}(z) \text{ is bounded }\}.$ Then $$U_a^+= \bigcup_{n=0}^{\infty} S_a^{ -n}(V_{R}^+) \text{ and } K_a^+ \subset V_{R} \cup V_{R}^-.$$ Also, from (1) it follows that $U_a^+$ is connected open set in $\mbb C^k.$

\smallskip
\item For any $z \in \mbb C^k$ and a fixed $R \ge R_a$ there exists $n_z \ge 0$ such that either of the following is true:
\begin{itemize}
\item[(i)] $S_a^{ n}(z) \in V_{R}$ for every $n \ge n_z$.

\smallskip
\item[(ii)] $S_a^{ n}(z) \in V_{R}^+$  for every $n \ge n_z$ and $S_a^{ n}(z) \to \infty$ as $n \to \infty.$
\end{itemize}
Note that (ii) does not say that $S_a^{ n}(z) \cap V_{R}= \emptyset$ for $z \in U_a^+$ and $n \ge 0$. However, it says that a point in $U_a^+$ will eventually belong to $V_{R}^+$ under the iterations of $S_a.$ The orbit of $z$ might either intersect or not intersect $V_{R}.$

\smallskip
\item Since all other points eventually land up in $V_{R}$, from dynamical point of view it is sufficient to study the behaviour of $S_a$ in $V_{R}.$ 

\smallskip
\item The choice of the radius of filtration $R$ can be made independent of the choice of $a$, provided $a$ comes from a bounded subset of $\mbb C^*$, i.e., for $0<|a|<A$ there exists $R_A \gg 1$ such that if $R \ge R_A$, the sets $V_R$, $V_{R}^+$ and $V_{R}^-$ serves as filtration for every $S_a$.

\smallskip
\item Let $K_a^\pm=\{ z \in \mbb C^k: S_a^{ \pm n}(z) \text{ is bounded as }n \to \infty \}.$ Then $K_a^+ \subset V_R \cup V_R^-$ and $K_a^- \subset V_R \cap V_R^+. $

\medskip
\item Now let
\[ J_a^+=\partial K_a^+, \; J_a^-=\partial K_a^- \text{ and } J_a=J_a^+ \cap J_a^-.\] Then the following are true:
\begin{itemize}
\item $S_a(K_a^+)=S_a^{-1}(K_a^+)=K_a^+$ and $S_a(K_a^-)=S_a^{-1}(K_a^-)=K_a^-.$

\smallskip
\item $S_a(J_a^+)=S_a^{-1}(J_a^+)=J_a^+$ and $S_a(J_a^-)=S_a^{-1}(J_a^-)=J_a^-.$

\smallskip
\item $S_a(J_a)=J_a=S_a^{-1}(J_a)$ and $J_a \subset V_R.$
\end{itemize}
\end{enumerate}
 
 \medskip\no
\begin{prop}\label{shift-regular}
Let $S_a$ be a polynomial shift--like map of type $\nu$ in $\mbb C^k.$ Then the $\eta-$th iterate of $S_a$ is regular where $\eta=\nu(k-\nu).$
\end{prop}
\begin{proof}
Let $d$ be the degree of $p.$

\begin{align*}
S_a^\nu \cord &=\big(z_{\nu+1},\hdots,z_k, az_1+p(z_{k-\nu+1}),az_2+p(z_{k-\nu+2}), \hdots , az_\nu+p(z_k) \big)\\
S_a^{-(k-\nu)}\cord &= \big(a^{-1}\{z_{\nu+1}-p(z_1)\},\hdots, a^{-1}\{z_k-p(z_{k-\nu})\}, z_1,\hdots,z_{\nu}\big).
\end{align*}
Let $\cal{S}_a^{m}$ and $\cal{S}_a^{-m}$ denote the map corresponding to $S_a^{m\nu}$ and $S_a^{-m(k-\nu)}$ in $\mbb P^{k+1}$ respectively for $m \ge 1$. Then
\begin{align*}
\cal{S}_a^{1}[z:w]=&[z_{\nu+1}w^{d-1}:\cdot\cdot:z_kw^{d-1}:z_{k-\nu+1}^d+wh_{k-\nu+1}^1(z,w):\cdot\cdot:z_{k}^d+wh_{k}^1(z,w): w^d]\\
\cal{S}_a^{-1}[z:w]=&[z_1^d+w\tilde{h}_1^1(z,w):\cdot\cdot:z_{k-\nu}^d+w\tilde{h}_{k-\nu}^1(z,w):w^{d-1}z_1\cdot\cdot:w^{d-1}z_{\nu}: w^d]
\end{align*}
where degree of $h_{k-\nu+i}^1< d$ and degree of $\tilde{h}_{j}^1< d$ for every $1 \le i \le \nu$ and $1 \le j \le k-\nu$ respectively.
For $m \ge 2$
\begin{align*}
\cal{S}_a^{m}[z:w]=&[wh_1^m(z,w):\cdot\cdot:z_{k-\nu+1}^{d^m}+wh_{k-\nu+1}^m(z,w):\cdot\cdot:z_{k}^d+wh_{k}^m(z,w): w^{d^m}]\\
\cal{S}_a^{-m}[z:w]=&[z_1^{d^m}+w\tilde{h}_1^m(z,w):\cdot\cdot:z_{k-\nu}^{d^m}+w\tilde{h}_{k-\nu}^m(z,w):w\tilde{h}_{k-\nu+1}^m(z,w)\cdot\cdot:w\tilde{h}_{k}^m(z,w): w^{d^m}]
\end{align*}
where degree of $h_{i}^m, \; \tilde{h}_i^m< d^m$  for every $1 \le i \le k.$ 

\medskip\no Hence the indeterminacy sets $I^+_m$ and $I^-_m$ of $\cal{S}_a^m$ and $\cal{S}_a^{-m}$  is given by
\begin{align*}
I^+=I^+_m&= \{[z_1:\cdots:z_{k-\nu}:0:\cdots:0:0]: z_i \neq 0 \text{ for some }i, \; 1 \le i \le k-\nu\}\\
I^-=I^-_m&= \{[0:\cdots:0: z_{k-\nu+1}:\cdots:z_k:0]: z_i \neq 0 \text{ for some }i, \;k-\nu+1 \le i \le k\}
\end{align*} 
for every $m \ge 1.$ Note that $I^+ \cap I^-=\emptyset$, thus $S_a^{\nu(k-\nu)}$ is regular.
\end{proof}
\begin{lem}\label{k-}
Let $S_a$ be a polynomial shift--like map of type $1 \le \nu\le k-1$ in $\mbb C^k.$ Then for $0<|a|<1$, $int(K_a^-)=\emptyset$, i.e., $K_a^-=J_a^-.$
\end{lem}
\begin{proof}
Recall that for $R>0$ the sets $V_R$, $V_R^+$ and $V_R^-$ gives a filtration for $S_a.$ Now for every $n \ge 1$ let $S_n$ be the sets defined as $$S_n=K_a^- \setminus S_a^n(V_R^+).$$ Since $K_a^- \setminus V_R^+$ is contained in a $V_R$ and $S_a(V_R^+) \subset V_R^+$, $S_n \subset S_{n+1}$ for every $n \ge 0$, i.e., $S_n$ is an increasing collection. If the Lebesgue measure of $S_0$, i.e., $m(S_0) \neq 0$, then $m(S_n)=|a|^n m(S_0)$ or $m(S_{n+1})< m(S_n)$ which is contradiction! Hence $m(S_0)=m(S_n)=0$ for every $n \ge 0.$ Now $K_a^- \subset V_R \cup V_R^+$, i.e., $m(K_a^-\setminus  V_R)+m(K_a^- \setminus  V_R^+)=0$ or $m(K_a^- \setminus  V_R)=0.$

\medskip\no 
\textit{Claim: } $m(K_a^- \setminus K_a)=0.$

\medskip\no 
Suppose $m(K_a^- \setminus K_a)\neq 0$ then there exists $z \in K_a^- \setminus K_a$ such that $B_{\ep}(z) \subset K_a^- \setminus K_a$ for some $\ep>0.$ Hence $B_{\ep}(z) \subset U^+_a$ or $S_a^n(B_{\ep}(z)) \subset V_R^+ \subset K_a^- \setminus V_R $ for sufficiently large $n$, i.e., $m(K_a^-\setminus  V_R)\ge |a|^n\ep>0$ which is a contradiction!

\medskip\no 
Also $m(K_a)=|a|m(K_a)=m(K_a)$, as $K_a$ is a compact set completely invariant under $S_a$. But $|a|<1$, hence $m(K_a)=0.$ Thus $m(K_a^-)=0.$
\end{proof}
\section{Proof of Theorem \ref{degeneration}}
\no Let $p$ be a monic polynomial in one variable of degree $d \ge 2$ . Recall that the Green function and the equilibrium measure associated to $p$ is defined as:
\[ G_p(z) =\lim_{n \to \infty} \frac{\log^+\|p^n(z)\|}{d^n}  \text{ and } \mu_p=\frac{1}{2 \pi}dd^c {G_p}.\]
Further, recall from Section \textbf{1}, that the notation $p_\nu$ was used to denote a map of the form
\[ p_{\nu}(z_1,\dots,z_{\nu})=(p(z_1),\hdots,p(z_{\nu}))\]
in $\mbb C^ \nu$, $\nu \ge 1.$
\begin{lem}\label{Step 1}
Let $G_p$ and $\mu_p$ be the Green function and equilibrium measure corresponding to the monic polynomial $p.$ Then for $\nu \ge 1$
\[\mu_{p_\nu}=\underbrace{\mu_p \wedge\cdots\wedge \mu_p}_{\nu-\text{times}}\]
and $\mu_{p_\nu}$ is the complex equilibrium measure on $K_p^\nu \subset \mbb C^\nu$ where $K_p$ is the filled Julia set for $p.$
\end{lem}
\begin{proof}
Note that $G_{p_{\nu}}\cord=\max\{G_p(z_i):1 \le i \le k\}$, i.e., $G_{p_\nu}$ is the extremal function for $K_p^\nu.$ Now by Proposition 2.2 from \cite{BT2}, the result follows.
\end{proof}
\begin{lem}\label{Step 2}
Let $S_a$ denote a $\nu-$shift of the polynomial $p.$ Then there exists a graph in $\nu$ variables (say $\Gamma_\nu$) such that $S_0^{k-\nu}(\mbb C^k)=\Gamma_\nu$ where 
\[ S_0\cord=(z_2,\hdots,z_k,p(z_{k-\nu+1})).\]
\end{lem}
\begin{proof} 
\textit{Case 1:} Suppose $\nu \ge k-\nu.$ Then consider
\[ \Gamma_{\nu}=\big(z_{1},\hdots,z_{\nu},p(z_1),\hdots,p(z_{k-\nu})\big) .\]
Now
\[ S_0^{k-\nu}\cord=\big (z_{k-\nu+1},\hdots,z_k,p(z_{k-\nu+1}),\hdots,p(z_{2(k-\nu)})\big). \]Hence the proof.

\medskip\no 
\textit{Case 2: }Suppose $\nu< k-\nu$. Let $l \ge 1$ such that $k-\nu=l \nu +r$ where $0 \le r < \nu.$ Consider $$\Gamma_\nu=(z_{1},\hdots,z_\nu,p(z_1),\hdots,p(z_\nu),\hdots,p^{l+1}(z_1),\hdots,p^{l+1}(z_r)).$$
 Now $S_0^{k-\nu}\equiv S_0^{l\nu+r}$ and 
 \small{$$S_0^{l\nu+r}\cord=\big(z_{k-\nu+1},\hdots, z_k,\hdots,p^l(z_{k-\nu+1}),\hdots,p^l(z_k),p^{l+1}(z_{k-\nu+1}),\hdots,p^{l+1}(z_{k-\nu+r})\big).$$}
 Hence $S_0^{k-\nu}(\mbb C^k)=\Gamma_\nu.$ 
\end{proof}
\begin{rem}\label{Step 3}
For $\nu \ge k-\nu$ 
\[ S_0^\nu\big(z_{2 \nu-k+1},\hdots,z_{\nu},p(z_1),\hdots,p(z_{\nu})\big)=\big(p(z_{2 \nu-k+1}),\hdots,p(z_{\nu}),p^2(z_1),\hdots,p^2(z_{\nu})\big).\]
Otherwise
\begin{align*}
&S_0^\nu\big(z_{1},\hdots,z_\nu,p(z_1),\hdots,p(z_\nu),\hdots,p^{l+1}(z_1),\hdots,p^{l+1}(z_r)\big)\\ &=\big(p(z_{1}),\hdots,p(z_\nu),p^2(z_1),\hdots,p^2(z_\nu),\hdots,p^{l+2}(z_1),\hdots,p^{l+2}(z_r)\big)
\end{align*}
Thus $S_0^{\nu}(\Gamma_{\nu}) =\Gamma_{\nu}.$
\end{rem}
\begin{rem}\label{Step 4}
The map $S_0^{\nu}$ on the graph $\Gamma_{\nu}$ is equivalent to the map $p_{\nu}$ in $\mbb C^\nu$, i.e., the following diagram commutes
\[
\begin{tikzcd}
\mbb C^\nu \arrow{r}{p_\nu}\arrow[swap]{d}{\phi}&\mbb  C^\nu \arrow{d}{\phi}\\
\Gamma_{\nu} \arrow{r}{S_0^{\nu}} & \Gamma_\nu
\end{tikzcd}.
\]
where $\phi$ is the map to the graph $\Gamma_\nu.$
\end{rem}
\begin{lem}\label{Step 5}
Let $G_a^+$ denote the positive Green function for $S_a$, i.e., for $z \in \mbb C^k$
\[ G_a^+(z)=\lim_{n \to \infty} \frac{\log^+\|S_a^{\nu n}(z)\|}{d^n}.\]
Then 
\[ \mu_0=\big(dd^c({G_0^+}_{|\Gamma_\nu})\big)^{\nu}=d^{k-\nu}(2\pi)^{\nu}\pi_*(\mu_{p_\nu})\]
where $\pi$ is the projection of $\mbb C^k$ to the last $\nu-$coordinates. 
\end{lem}
\begin{proof}
Recall that for $z \in \mbb C^k$
\[ G_0^+(z)=\lim_{n \to \infty}\frac{\log^+\|S_0^{\nu n}(z)\|}{d^n}.\]
From the filtration of $S_a$ it follows that
\[ G_0^+(z)=G_{p_{\nu}}(\pi(z))\]
\no \textit{Case 1: }For $\nu \ge k-\nu$ and $z \in \Gamma_\nu$
\begin{align*}
G_0^+(z)&=G_{p_{\nu}}\big(z_{k-\nu+1},\hdots,z_\nu,p(z_1),\hdots,p(z_{k-\nu})\big)\\
&=\max\{dG_p(z_1),\hdots,dG_p(z_{k-\nu}),G_p(z_{k-\nu+1}),\hdots,G_p(z_{\nu})\}\\
&=G_{p_\nu}\big(p(z_1),\hdots,p(z_{k-\nu}),z_{k-\nu+1},\hdots,z_\nu\big).
\end{align*} 
Now from Lemma \ref{Step 1}
\[\mu_0= \big(dd^c({G_0^+}_{|\Gamma_\nu})\big)^{\nu} =d^{k-\nu}(2 \pi)^\nu \phi_*(\mu_{p_\nu}).\]
\textit{Case 2: }For $\nu < k-\nu$ and $z \in \Gamma_\nu$
\begin{align*}
G_0^+(z)&=G_{p_{\nu}}\big(p^l(z_{r+1}),\hdots,p^l(z_\nu),p^{l+1}(z_1),\hdots,p^{l+1}(z_{r})\big)
\end{align*}
where $k-\nu=l \nu+r$ for $l \ge 1$ and $0 \le r < \nu. $ Arguing as in \textit{Case 1} and using Lemma \ref{Step 1}, it follows that
\begin{align*}
 \mu_0=\big(dd^c({G_0^+}_{|\Gamma_\nu})\big)^{\nu}&=d^{l(\nu-r)+(l+1)r}(2 \pi)^\nu \phi_*(\mu_{p_\nu})\\
 &=d^{k-\nu}(2 \pi)^\nu \phi_*(\mu_{p_\nu}).
 \end{align*}
\end{proof}
\begin{lem}\label{Step 6}
For any compact set $K \in \mbb C^k \setminus \Gamma_\nu$, there exists $A_K>0$ such that if $0<|a|<A_K$, then $S_a^{-(k-\nu)}(K) \subset V_R^-$ where $R$ is the radius of filtration whenever $|a|<1$
\end{lem}
\begin{proof}
Recall that
\begin{align*}
S_a^{-(k-\nu)}\cord &= \big(a^{-1}\{z_{\nu+1}-p(z_1)\},\hdots, a^{-1}\{z_k-p(z_{k-\nu})\}, z_1,\hdots,z_{\nu}\big).
\end{align*}
Pick a $z \in K$. Since $z \notin \Gamma_\nu$ it follows that 
\[ c(z)=\max\{|z_{\nu+i}-p(z_i)|: 1 \le i \le k-\nu\} \neq 0.\]
As $K$ is compact there exists a constant $C>0$ such that $c(z)>C$ for every $z \in K.$ Further, choose $M>0$ such that $K \subset B(0;M)$, i.e., the ball of radius $M$ at the origin. Now choose $0<A<1$ sufficiently small such that $A^{-1}C> \max\{M,R\}.$ This complete the proof.
\end{proof}
\begin{rem}\label{Step 7}
Recall Theorem 9 from \cite{BP}. The negative Green function of $S_a$ is given by
\[ G_a^-(z)=\lim_{n \to \infty} \frac{\log^+\|S_a^{-(k-\nu) n}(z)\|}{d^n}.\] 
Further, on $V_R^-$ 
\[ G_a^-(z)=\log\|z\|+ o(1)\]
and $o(1) \to 0$ uniformly as $z \to \infty.$
\end{rem}
\begin{lem}\label{Step 8}
Given any compact set $K $ in $\mbb C^k \setminus \Gamma_{\nu}$
\[ \lim_{a \to 0} G_a^-(z)+\frac{\log|a|}{d}=\frac{1}{d}\Big\{\log\big(\max\{|z_{\nu+i}-p(z_i)|: 1 \le i \le k-\nu\}\big)\Big\}\]
whenever $z \in K.$
\end{lem}
\begin{proof}
From Remark \ref{Step 7}, if $z \in V_R^-$ then $\|z\|_{\infty}=\max\{|z_i|: 1\le i \le k-\nu\}$, i.e., 
\begin{align} \label{(1)}
G_a^-(z)=\log\big(\max\{|z_i|: 1\le i \le k-\nu\}\big)+ o(1)
\end{align}
and $o(1) \to 0$ uniformly as $z \to \infty.$

\medskip\no 
Also from definition of $G_a^-$, it follows that for $z \in \mbb C^k$
\begin{align}\label{(2)}
G_a^-(S_a^{-(k-\nu)}(z))=d G_a^-(z).
\end{align}
Now from Lemma \ref{Step 6}, it follows that $S_a^{-(k-\nu)}(K) \subset V_R^-$ whenever $0<|a|<A_K.$ Thus from (\ref{(1)}) and \ref{(2)} it follows that
\begin{align*}
&G_a^-(z)=\frac{1}{d} \bigg(\log\Big(\max\{|a^{-1}\big(z_{\nu+i}-p(z_i)\big)|: 1\le i \le k-\nu\}\Big)\bigg)+ o(1)\\
&G_a^-(z)+\frac{\log|a|}{d}=\frac{1}{d} \Big\{\log\big(\max\{|z_{\nu+i}-p(z_i)|: 1\le i \le k-\nu\}\big)\Big\}+ o(1)
\end{align*}
Also $o(1) \to 0$ uniformly as $a \to 0$, thus the proof.
\end{proof}
\begin{rem}\label{Step 9}
Thus for $z \in \mbb C^k$, the function $H_a$ defined as
\[ H_a(z)=G_a^-(z)+\frac{\log|a|}{d}\]
converges in $L^1_{\text{loc}}$ to 
\[F(z)=\frac{1}{d} \Big\{\log\big(\max\{|z_{\nu+i}-p(z_i)|: 1\le i \le k-\nu\}\big)\Big\}\]
as $a \to 0.$
\end{rem}
\begin{rem}\label{Step 10}
For $a \in \mbb C^{k-\nu}$(see \cite{Demailly}), 
\begin{align}\label{PL}
(dd^c \log\{\max|z_i-a_i|: 1 \le i \le k-\nu\})^{k-\nu}=(2 \pi)^{k-\nu}\delta_a.
\end{align}
Now consider $\cal{F}: \mbb C^k \to \mbb C^{k-\nu}$ as
\[ \cal{F}\cord=\big(z_{\nu+1}-p(z_1),\hdots,z_{k}-p(z_{k-\nu})\big).\]
From (\ref{PL}), it follows that 
\[(dd^c\log|\cal{F}|)^{k-\nu}=(2\pi)^{k-\nu}[\Gamma_{\nu}],\]i.e.,
\[ (dd^cF)^{k-\nu}=d^{\nu-k}(2\pi)^{k-\nu}[\Gamma_{\nu}].\]
where $[\Gamma_\nu]$ denotes the current of integration on $\Gamma_\nu.$
\end{rem}
\no Recall from Theorem 11 in \cite{BP}, $\mu_a^-=(\frac{1}{2\pi}dd^c G_a^-)^{k-\nu}$ and $\mu_a=(\frac{1}{2\pi}dd^c G_a^+)^\nu.$ Also
\[ \mu_a=\mu_a^+ \wedge \mu_a^-=(\frac{1}{2\pi}dd^c G_a)^{k}\]
where $G_a=\max(G_a^+,G_a^-).$ Also ${S_a}_*^{\nu(k-\nu)}(\mu_a)=\mu_a.$

\medskip\no Let $\Lambda_a$ denote the Lyapunov exponent of $S_a^{\nu(k-\nu)}$ with respect to the measure $\mu_a$, i.e.,
\[ \Lambda_a=\lim_{n \to \infty} \frac{1}{n} \int \|DS_a^{\nu(k-\nu)n}\|d \mu_a.\]
\begin{lem}\label{Lyapunov}
$a \to \Lambda_a$ is a upper semi continuous function.
\end{lem}
\begin{proof}
The proof is similar to Lemma 5.1 in \cite{BS3}. For a fixed $a \in \mbb C$, let $$\Lambda_a^n=\frac{1}{n} \int \|DS_a^{\nu(k-\nu)n}\|d \mu_a.$$
Now as ${S_a}_*^{\nu(k-\nu)}(\mu_a)=\mu_a$ it follows that 
\[ (m+n)\Lambda_a^{n+m} \le n \Lambda_a^n+m\Lambda_a^m.\]
From Theorem 4.9 in \cite{WaltersBook}, $\Gamma_a$ exists. Now for $n \ge 1$ consider the sequence $f_n(a)=\Lambda_a^{2^n}.$ Then $f_{n+1}(a) \le f_n(a)$, i.e., $\Lambda_a$ is a limit of a decreasing sequence of continuous function, thus $a \to \Lambda_a$ is upper semi continuous.
\end{proof}
\begin{proof}[Proof of Theorem \ref{degeneration}]
Note that $dd^c H_a=dd^c G_a^-$ where $H_a$ is as defined in Remark \ref{Step 9}. From Remark \ref{Step 10}, $(dd^c H_a)^{k-\nu}$ converges to $d^{\nu-k}[\Gamma_{\nu}]$ in the sense of currents as $a \to 0.$
Hence from Lemma \ref{Step 5}
\begin{align*}
\lim_{a \to 0}\mu_a&=\lim_{a \to 0}\mu_a^+ \wedge \lim_{a \to 0} \mu_a^- \\
&=\lim_{a \to 0} (\frac{1}{2\pi}dd^c G_a^+)^\nu \wedge \lim_{a \to 0} (\frac{1}{2\pi}dd^c G_a^-)^{k-\nu}\\
&=\lim_{a \to 0} (\frac{1}{2\pi}dd^c G_a^+)^\nu  \wedge d^{\nu-k}[\Gamma_{\nu}].\\
&=d^{\nu-k}(\frac{1}{2\pi}dd^c {G_0^+}_{|\Gamma_\nu})^\nu=\pi_*(\mu_{p_\nu}).
\end{align*}
Now the Lyapunov exponent for $p_\nu$ with respect to the measure $\mu_{p_\nu}$ is given by 
\[\Lambda_{p_\nu}=\lim_{n \to \infty} \frac{1}{n} \int \|Dp_\nu^{n}\|d \mu_{p_\nu}.\]
Since Lyapunov exponents are conjugacy invariant and $a \to \Lambda_a$ is upper semi continuous
\begin{align*}
(k-\nu)\Lambda_{p_\nu}&=\lim_{n \to \infty} \frac{k-\nu}{n} \int \|Dp_\nu^{n}\|d \mu_{p_\nu} \\
&=\lim_{n \to \infty} \frac{1}{n} \int \|Dp_\nu^{(k-\nu)n}\|d \mu_{p_\nu}\\
&=\lim_{n \to \infty} \frac{1}{n} \int \|DS_0^{\nu(k-\nu)n}\|d \mu_0\\
&=\limsup_{a \to 0} \Lambda_a
\end{align*} 
\end{proof}
\section{Hyperbolic polynomials in one variable}
\no In this section, we will recall a few properties of hyperbolic polynomials in one variable (see \cite{McMullen}). For a rational function $f$ in $\hat{\mbb C}$, let  $J_f$ and $K_f$ be the usual Julia set  and filled Julia set of $f$ respectively. Also, let $C_f$ denote the critical points of $f $ in $\hat{\mbb C}$.
\begin{defn} 
A rational function $f$ is said to be hyperbolic if the post critical set of $f$ does not intersect the Julia set of $f$, i.e.,
\[P_c (f) \cap J_f=\emptyset\] where $$P_c(f)=\ov{\bigcup_{n=0}^{\infty} f^{ n}(C_f)}.$$
\end{defn}
\no Let us recall a few properties of a hyperbolic polynomial (say $p$) of $\mbb C$ which will be essential for our work. Also we will assume that $p$ has a connected Julia set, i.e., the Fatou component at infinity is simply connected.
\begin{enumerate}
\item  Let $J_p$ be the Julia set for a hyperbolic polynomial. Then $J_p$ is compact and there exists a neighbourhood of $J_p$, i.e., $V \supset J_p$ such that $V$ admits a conformal metric for which $p$ is expanding.

\medskip
\item There exists a constant depending on $p$, i.e., $\delta(p)>0$ such that $$U_{\delta(p)} \subset\subset p(U)\text{ and }p^{-1}(U)_{\delta(p)} \subset \subset U,$$ where $U_{\delta(p)}$ and $p^{-1}(U)_{\delta(p)}$ are $\delta(p)-$neighbourhoods of $U$ and $p^{-1}(U).$

\medskip 
\item Now let $U$ be a slightly smaller neighbourhood of $J_p$ contained in $V.$ The complement of $U$, i.e., $U^c=U_c \cup U_\infty$ where $U_c$ denote all the compact components of $U^c$ and $U_\infty$ the unbounded component containing infinity. Then number of components in $U_c$ is finite and $p(U_c) \subset U_c$.

\medskip
\item Note that $U_c$ may be empty as well. If non empty, then the components of $U_c$ are simply connected domains, compactly contained in the basin of attraction of periodic points of $p.$ Thus with respect to the hyperbolic metric (say $\mbb H$) on the basin of attraction of periodic points of $p$, there exists a constant $0<c<1$, such that for $v \in T_z{U_c}$, $z \in U_c$
\[ \|p'(z)v\|_{\mbb H}^{p(z)}<c\|z\|_{\mbb H}^z.\]

\medskip
\item The component $U_{\infty}$ is simply connected and there exists a conformal map $\rho$,  $$\rho:\hat{\mbb C}\setminus \bar{\Delta}\to\hat{\mbb{C}} \setminus K_p$$ such that $\rho(z^d)=p \circ \rho(z)$ where $\bar{\Delta}$ is the unit disc and $d \ge 2$ is the degree of the polynomial. 

\medskip
\end{enumerate}

\begin{prop}\label{U-infinity}
Suppose $p$ is a hyperbolic polynomial with connected Julia set then there exists $\eta(p)>0$ such that if $\{w_n\} \subset  D(0;\eta(p))$ and $z_0 \in U_\infty$ the sequence $\{z_n\}$ defined as 
\[ z_{n}=p(z_{n-1})+w_n\] diverges uniformly to infinity.
\end{prop}

\begin{proof}
Our aim is to show that for a given $R>0$, sufficiently large $|z_n|>R$ for every $n \ge n_0 \ge 1.$ Further, the existence of $n_0$ depends only on $R$ and is independent of the choice $z_0 \in U_{\infty}$ or the sequence $\{w_n\}.$

\medskip\no 
Given a polynomial $p$ there exists $R_p>0$ such that for $|z|\ge R_p$, $|p(z)|-|z|>\ep_p>0$, since infinity is a super attracting fixed point for $p.$ Hence if $z_0 \in U_\infty \cap D(0;R_p)^c$ then for every sequence $\{w_n\} \in D(0; \ep_p)$, $|z_n| >R_p+n \ep_p> R$ for $n \ge n_0.$

\medskip\no 
Without loss of generality let us assume that $K=\ov{U}_{\infty} \cap \ov{D(0;R_p)}$ be a non--empty compact subset of $\mbb C\setminus K_p.$ By Property (4) of hyperbolic polynomials $K_1=\rho^{-1}(K)$ is a compact subset of $\mbb C \setminus \bar{\Delta}.$

\medskip\no 
For any $w \in K_1$, then $|w|>1+\ep$ for some $\ep>0.$ Choose $r_1>0$ such that $|w|^d-r_1>1+d \ep$ for every $w \in K_1.$ Let $$M_1=\min\{|w|^d-r_1:w \in K_1\}.$$ Then $M_1> 1+d \ep.$ Similarly one can choose $r_2$ such that $M_1^d-r_2> 1+d^2 \ep.$ Let $M_2=M_1^d-r_2.$ Inductively one can choose $r_n>0$ for every $n \ge 1$ such that $$M_n>1+d^n \ep$$ where $M_n=M_{n-1}^d-r_n$. Thus $M_n \to \infty$ as $n \to \infty$, i.e., there exists $n_0 \ge 0$ such that if $\{\tilde{w}_n\}$ is a sequence from $D(0;\tilde{r})$ where $\tilde{r}=\min\{r_1: 1\le i \le n_0\}$ and $\tilde{z}_0 \in K_1$, $(n_0+1)-$th element in the  sequence $\{\tilde{z}_n\}$ defined as 
\[ \tilde{z}_{n}=\tilde{z}_{n-1}^d+\tilde{w}_n\] leaves $K_1.$ Note that this choice of $n_0$ is independent of the choice of base point $\tilde{z}_0 \in K_1$ or the sequence $\{w_n\}$ in $D(0; \tilde{r}).$

\medskip\no 
Let $m=\min\{|z|: z \in K_1\}$ and $M=\max\{|z|: z \in K_1\}$. Then $K_1 \subset \ov{\cal{A}(0;m;M)} \subset \mbb C \setminus \bar{\Delta}$ (the closed annulus with inner radius $m$ and outer radius $M$). Let $K_{\cal{A}}=\rho(\ov{\cal{A}(0;m;M)}).$ By continuity for $\tilde{r}>0$ (as chosen above) there exists $r_0>0$ such that 
\[ |\rho^{-1}(z)-\rho^{-1}(w)|< \tilde{r} \text{ whenever } |z-w|< r_0\]
and $z,w \in K_{\cal{A}}.$
Let $\eta(p)=\min\{r_0,\ep_p\}.$ Further, let $\{w_n\}$ be any sequence from  $D(0; \eta(p))$ and $z_0$ be an arbitrary point in $K.$ By definition
\begin{align*}
z_1=p(z_0)+w_1 
\end{align*}
i.e.,  \[|z_1-\rho \circ (\rho^{-1}(z_0))^d|=|w_1|< r_0.\]
Hence $\rho^{-1}(z_1)-(\rho^{-1}(z_0))^d=\tilde{w}_1$ where $\tilde{w}_1 \in D(0; \tilde{r}).$ 

\medskip\no Let $\tilde{z}_n=\rho^{-1}(z_n)$ for every $n \ge 0.$ 

\medskip\no 
\textit{Case 1: }If $|\tilde{z_1}|>M$, then $|z_1|>R_p$ and $z_n=p(z_{n-1})+w_n$ lies in $D(0;R_p)^c$ for every $n \ge 2$ and $z_n \to \infty$ as $n \to \infty.$

\medskip\no 
\textit{Case 2: }If $|\tilde{z_1}|< M$ then $z_1 \in K_{\cal{A}}.$ By a similar argument as before it follows that
\[ \tilde{z}_2=\tilde{z}_1^d+\tilde{w}_2 \text{ where } \tilde{w}_2 \in D(0; \tilde{r}).\]
Thus again we do the same analysis for $\tilde{z}_2$ and continuing inductively we have that there exists $1 \le N_0 \le n_0$ such that  $|\tilde{z}_{N_0+1}|> M.$ Now by \textit{Case 1}, $z_n \to \infty$ as $n \to \infty.$
\end{proof}

\section{Proof of Theorem \ref{main theorem 1} and \ref{main theorem 2} for \textit{k} = 3}
\no Let $p$ be hyperbolic polynomial with connected Julia set. Let $V$ be the neighbourhood of $J_p$ that admits a conformal metric $\rho$, which is expanding under $p$ and $U$ be (as before) a neighbourhood of $J_p$ and compactly contained in $V$ such that there exists $\delta>0$ for which $U_{\delta} \subset p(U) $ and $p^{-1}(U)_{\delta} \subset U.$ Let $U_c$ and $U_\infty$ denote the the compact components and the component containing infinity in the complement of $U$ and let $U_0=\bar{U}$. Further, let $A$ be sufficiently small such that $AR< \delta_0$, where $\delta_0< \max\{\delta,\eta(p)\} $ where $\eta(p)$ as in Proposition \ref{U-infinity} and $R>0$ be the radius of filtration whenever $|a|<1.$
\subsection*{Proof of Theorem \ref{main theorem 1}}
Let $S_a$ denote a type $1-$shift map of this hyperbolic polynomial in $\mbb C^3$ for $a \in \mbb C$, i.e.,
\[ S_a(z_1,z_2,z_3)=(z_2,z_3,az_1+p(z_3)).\]
\begin{lem}\label{size of Julia}
For $0<|a|<A$, the Julia set $J_a \subset D_R \times D_R \times U_0.$
\end{lem}
\begin{proof}
Note that for $z \in \mbb C^3 \cap V$ if $\pi_3(z)=z_3 \in U_c$ or $U_\infty$, then for sufficiently small choice of $|a|$, $$S_a^2(z) \in U_c \times U_c \times U_c \text{ or }U_{\infty} \times U_{\infty} \times U_{\infty}$$ by Lemma \ref{U-infinity}. Hence $S_a^{n}(z)$ is either uniformly bounded or diverges to infinity uniformly on a neighbourhood of $z$, i.e., $z \notin J_a^+.$ Thus the proof.
\end{proof}
\no When $a=0$, i.e., $S_0(z_1,z_2,z_3)=(z_2,z_3,p(z_3))$ and 
\[ DS_0^2(z_1,z_2,z_3)= \begin{pmatrix}
0 &0 &1 \\ 0 &0 &p'(z_3) \\ 0 &0 &(p^2)'(z_3)
\end{pmatrix}.\] Let the eigenvalues of $DS_0^2(z)$ be denoted by $\la_i^z(0)=0, 0\text{ and }p'(z_3)$ for $i=1,2,3$ and
\[ E_i^z(0)=\text{The eigenspace corresponding to the eigenvalue }\la_i^z(0) .\]
Then the eigenspaces $E_i^z(0)$ are as follows:
\begin{align*}
&E_1^z(0)=\{(t,0,0): t \in \mbb C\}, \\
&E_2^z(0)=\{(0,t,0): t \in \mbb C\}, \\
&E_3^z(0)=\{\big(t,tp'(z_3),t(p^2)'(z_3)\big): t \in \mbb C\}.
\end{align*}
Let $\Gamma$ denote the following graphs in $\mbb C^3$
\[ \Gamma=\{(z,p(z),p^{2}(z)): z \in \mbb C\}.\]
Observe that $S_0^{2}(\mbb C^3)=\Gamma$, i.e., $S_0^{2}(z_1,z_2,z_3)=(z_3,p(z_3),p^2(z_3)).$ Let the parametrization of $\Gamma$ be denoted by $\phi$, i.e., $\phi(z)=(z,p(z),p^{2}(z))$. Then for $v \in T_{\phi(z)} \Gamma$ means $v=(t,p'(z)t,(p^{2})'(z)t)$ for some $t \in \mbb C.$

\medskip\no 
\textit{Claim: }$E_3^z(0)=T_{S_0^{2}(z)} \Gamma$ and 
$DS_0^{2}(z)\big(E_3^z(0)\big)=E^{S_0^{2}(z)}_3(0).$

\medskip\no 
If $v \in E_3^z(0)$ for some $z \in \mbb C^3$ then $DS_0^{2}(z)v=(p^2)'(z_3)v$. Calculating coordinate wise it is directly obtained that $v=(t,p'(z_3)t,(p^{2})'(z_3)t)$ for some $t \in \mbb C.$ Hence, the first part of the claim is true.

\medskip\no 
Since $S_0^{2}(z)=(z_3,p(z_3),p^2(z_3))$, $E^{S_0^{2}(z)}_3(0)=T_{S_0^{4}(z)}\Gamma.$ But $S_0^{2}$ on the graph is equivalent to the map $p^2$ in $\mbb C$, i.e., the diagram is as follows
\[
\begin{tikzcd}
\mbb C^3 \arrow{r}{S_0^{2}} \arrow[swap]{d}{S_0^{2}} & \mbb C^3\arrow{d}{S_0^{2}} \\
\Gamma \arrow{r}{S_0^{2}}\arrow[swap]{d}{\phi} & \Gamma \arrow{d}{\phi}\\
\mbb C \arrow{r}{p^{2}} &\mbb C
\end{tikzcd}.
\]
Thus $D(E_3^z(0))=E_3^{S_0^{2}}(z)(0).$
%
\begin{proof}[Proof of Theorem \ref{main theorem 1} for $k=3$]
\no From Lemma \ref{size of Julia}, $J_a \subset D_R' \times D_R' \times V$ where $R'>R$ and $V$ is the neighbourhood of $J_p$ that compactly contains $U.$ Let $v=(t,tp'(z_3),t(p^2)'(z_3)) \in E_3^z(0)$. We consider the identification, $\phi_3(v)=t \in T_{z_3}V. $ Now $V$ admits a hyperbolic metric $\rho$ such that $p^2$ is expanding. Define $\|v\|_{\eta}=\|\phi_3(v)\|_{\rho}$, i.e., there exists a $\la>1$ such that
\begin{align}\label{expandingequation1}
\|DS_0^2(z)v\|^{S_0^2(z)}_{\eta}=\|(p^2)'(z_3)\phi_3(v)\|^{p^2(z_3)}_{\rho}> \la \|\phi_3(v)\|_{\eta}^{z_3}=\la \|v\|^{z}_{\eta}. 
\end{align}
For $z \in D_{R'} \times D_{R'} \times V$, and $v=(v_1,v_2,v_3)\in E_1^z(0) \oplus E_2^z(0) \oplus E_3^z(0) = T_z (D_{R'} \times D_{R'} \times V)$ we define the metric as:
\[ \|v\|_{\varrho_0}^z=\|v_1\|^z_{\mbb E}+\|v_2\|_{\mbb E}^z + \|v_3\|_{\eta}^{z}.\]
Let $E_u^z(a)= E_3^z(a)$ and $E_s^z(a)=E_1^z(a) \oplus E_2^z(a).$ From invariance of $E_s^z(0)$ and $E_u^z(0)$ and (\ref{expandingequation1}) it follows that there exists $\la_0>1$ such that:
\begin{align}\label{lambda_final_basic}
& \|DS_0^{2}(z)v^u\|_{\varrho_0}^{S_0^{2}(z)}> \la_0 \|v^u\|_{\varrho_0}^z\\
& \|DS_0^{2}(z)v^s\|_{\varrho_0}^{S_0^{2}(z)}< \la_0^{-1} \|v^s\|_{\varrho_0}^z .
\end{align}
Note that there exists a linear isometry for every point $z \in D_{R'} \times D_{R'} \times V$ as below
\[ \phi_{a,z}: T_z (D_{R'} \times D_{R'} \times V) \to T_z (D_{R'} \times D_{R'} \times V) \text{ and } \phi_{a,z}(v_1^0)=v_1^a, \phi_{a,z}(v_2^0)=v_2^a \]
where $(v_1^a,v_2^a) \in E_z^s(a)\oplus E_z^u(a). $ Also $\phi_{a,z}$ depends continuously on $a$ and $z$, hence for every $0 \le |a|< A$ there exists a conformal metric on $D_{R'} \times D_{R'} \times V$, $v \in E_z^s(a)\oplus E_z^u(a)$
\[ \|(v_1,v_2)\|_{\varrho_a}^z=\|\phi_{a,z}^{-1}(v_1,v_2)\|_{\varrho_0}.\]
For $\rho>0$ recall the definition of cones of $T_z(D_{R'} \times D_{R'} \times V) $ 
\begin{align*}
 C^s_z(\rho,a)&=\{(v_1^a,v_2^a) \in E_z^s(a)\oplus E_z^u(a): \|v_2^a\|_{\varrho_a}^z < \rho \|v_1^a\|_{\varrho_a}\} \text{ and }
\\ C^u_z(\rho,a)&=\{(v_1^a,v_2^a) \in E_z^s(a)\oplus E_z^u(a): \|v_1^a\|_{\varrho_a}^z < \rho \|v_2^a\|_{\varrho_a}\}.
\end{align*}
From (\ref{lambda_final_basic}), it follows that for sufficiently small $\rho>0$ there exists $\la_1>1$ such that for $z \in D_{R'} \times D_{R'} \times V$, $v \in C^u_z(\rho,0)$
\[ \|DS_0^{2}(z)v\|_{\varrho_0}^{S_0^{2}(z)}> \la_1 \|v\|_{\varrho_0}^z\]
and $v \in C^s_z(\rho,0)$
\[ \|DS_0^{2}(z)v\|_{\varrho_0}^{S_0^{2}(z)}< \la_1^{-1} \|v\|_{\varrho_0}^z .\]
Hence by continuity of splitting of the tangent spaces on $a$, the choice of $A$ is further modified, such that there exists $\rho_0>0$ with $C^s_z(\rho_0,a) \subset C^s_z(\rho,0)$ and $C^u_z(\rho_0,a) \subset C^u_z(\rho,0)$ for every $z \in \ov{D_{R'} \times D_{R'} \times V}$ whenever $0 \le |a|< A$, i.e., for
$z \in \ov{D_{R'} \times D_{R'} \times V}$, $v \in C^u_z(\rho_0,a)$
\[ \|DS_0^{2}(z)v\|_{\varrho_0}^{S_0^{2}(z)}> \la_1 \|v\|_{\varrho_0}^z\]
and $v \in C^s_z(\rho_0,a)$
\[ \|DS_0^{2}(z)v\|_{\varrho_0}^{S_0^{2}(z)}< \la_1^{-1} \|v\|_{\varrho_0}^z .\]
Again using the continuity of the metric on $a$, there exist modified $A$, $\rho_1$ (say $\rho_0 \ge \rho_1>0$) and $\la_1$ (say $\la_1>1$) such that 
 whenever $0< |a|< A$ for
$z \in \ov{D_{R'} \times D_{R'} \times V}$, $v \in C^u_z(\rho_1,a)$
\begin{align}\label{final3_hyp_1}
 \|DS_a^{2}(z)v\|_{\varrho_a}^{S_a^{2}(z)}> \la_1 \|v\|_{\varrho_a}^z
 \end{align}
and $v \in C^s_z(\rho_1,a)$
\begin{align}\label{final3_hyp_2}
 \|DS_a^{2}(z)v\|_{\varrho_a}^{S_a^{2}(z)}< \la_1^{-1} \|v\|_{\varrho_a}^z
\end{align}
Since $S_a$ is an automorphism further modifying the choice of $\rho_1$ and $A$, i.e., for $0<|a|<A$, equation (\ref{final3_hyp_1}) and (\ref{final3_hyp_2}) can be summarized as:
\begin{itemize}
\item[(i)] There exists a Riemannian metric $\|\cdot \|_{\varrho_a}$ on ${D_{R'} \times D_{R'} \times V} \subset \mbb C^3$ such that $$DS_a^{2}(z)(C^s_z(\rho_1,a)) \subset \text{int}\; C^s_{S_a^{2}(z)}(\rho_1,a) $$ and $$DS_a^{-2}(z)\big(C^u_{S_a^{2}(z)}(\rho_1,a)\big) \subset \text{int}\; C^u_{z}(\rho_1,a).$$

\medskip 
\item[(ii)] There exists $\la_1>1$ such that for every $z \in J_a $
\[ \|DS_a^{-2}(z)(v)\|_{\varrho_a}^{S_a^{-2}(z)} \ge \la_1 \|v\|_{\varrho_a}^z \text{ for } v\in C^s_{S_a^2(z)}(\rho_1,a)\]
and
\[\|DS_a^{2}(z)(v)\|_{\varrho_a}^{S_a^{2}(z)} \ge \la_1 \|v\|_{\varrho_a}^z \text{ for } v\in C^u_z(\rho_1,a).\]
\end{itemize}
Now by Corollary 6.4.8 in \cite{KatokBook} $S_a^2$ is hyperbolic on $J_a.$

\medskip\no 
Since the topology induced in $V$ by the conformal metric $\rho$ is equivalent to the Euclidean metric, the metric $\varrho_a$ induce the same topology as the general Euclidean metric on $D_{R'} \times D_{R'} \times V.$ Thus the proof.

\end{proof}
\subsection*{Proof of Theorem \ref{main theorem 2}}Let $S_a$ denote a type $2-$shift map of the hyperbolic polynomial in $\mbb C^3$, i.e., for $a \in \mbb C^*$
\[ S_a(z_1,z_2,z_3)=(z_2,z_3,az_1+p(z_2)).\]
\begin{lem}\label{lemma 1}
For $0<|a|<A$, the Julia set $J_a$ of $S_a$ is contained in the union of the following sets, i.e., 
\[ J_a \subset ({D_R} \times U_0 \times U_c )\cup (D_R \times U_c \times U_0 ) \cup (D_R \times U_0 \times U_0)\]
where $D_R$ is the disc of radius $R$ at the origin in $\mbb C$ and $R$ is the radius of filtration for $|a|<1.$
\end{lem}
\begin{proof}
Note that $J_a \subset V_R \subset D_R \times \mbb C^2$ and 
\begin{align}\label{partition}
D_R \times \mbb C^2= \bigcup_{i=0,c,\infty}\bigcup_{j=0,c,\infty}\bigcup_{k=0,c,\infty} D_R \times U_i \times U_j.
\end{align}
Suppose $z \in J_a$ and $\pi_3(z) \in U_{\infty}.$ Then 
$\pi_3 \circ S_a^2(z)=p(z_3)+az_2 \in U_{\infty}^{\circ}$ or $\pi_2 \circ S_a^3(z)\in U_{\infty}^{\circ}.$  Thus there  exists a neighbourhood of $z$ on which either 
$\pi_2 \circ S_a^{2n+1} \to \infty$ uniformly or $\pi_3 \circ S_a^{2n} \to \infty$ uniformly. This means $z \in F_a$ (i.e., the Fatou set of $S_a$) which is a contradiction! A similar argument gives that if $z \in J_a$ then $\pi_2 (z) \notin U_{\infty}.$ 
%

\medskip\no 
Further, if $z \in J_a$ and $\pi_i(z) \in U_c$ for every $i=2,3.$ Then $S_a^3(z) \in int(U_c \times U_c \times U_c).$ Hence the sequence $\{S_a^n\}$ is uniformly bounded on a neighbourhood of $z$, i.e., $z$ lies in the Fatou set of $S_a.$ 

\medskip\no 
Thus
\begin{align}\label{refined partition}
J_a\subset (D_R \times U_0 \times U_c )\cup (D_R \times U_c \times U_0 ) \cup (D_R \times U_0 \times U_0).
\end{align}
Thus the proof.
\end{proof}
\begin{rem}\label{open neighbourhood}
Note that $U_0=\bar{U}$ and $\partial U_0=\partial U_c \cup \partial U_\infty.$ Since $J_a \cap D_R \times U_c \times U_c=\emptyset$ and $J_a \cap D_R \times U_\infty \times U_\infty=\emptyset$,
\[J_a\subset (D_R \times U \times U_c )\cup (D_R \times U_c \times U) \cup (D_R \times U \times U),\] i.e.,
\[J_a\subset (D_R \times U \times W )\cup (D_R \times W \times U ) \] where $W=U \cup U_c$ and is an open neighbourhood of the filled Julia set $K_p$.
\end{rem}
\no Let $U_1=\ov{D_R} \times U_0 \times U_c$, $U_2=\ov{D_R} \times U_c \times U_0$ and $U_3=\ov{D_R} \times U_0 \times U_0.$
\begin{lem}\label{lemma 2}
Let 
\[ J_a^1=\bigcup_{n=0}^{\infty} S_a^{-2n}(U_1 \cap J_a), \; J_a^2=\bigcup_{n=0}^{\infty} S_a^{-2n}(U_2 \cap J_a), \text{ and } J_a^3={J_a \setminus {J_a^{1} \cup J_a^{2}}}.\] Then $J_a^1$, $J_a^2$ and $J_a^3$ are completely invariant under $S_a^2$.
\end{lem}
\begin{proof}
If $z \in J_a^{3}$ then $S_a^{2n}(z) \notin int(U_1)$ or $int(U_2)$ for every $n \ge 0$, i.e., 
\[ S_a^{2n}(z) \in U_3 \] for every $n \ge 0.$ 
Note that if $J_a^{1}$ and $J_a^{2}$ is completely invariant under $S_a^2$, invariance of $J_a^3$ follows from the invariance of $J_a.$

\medskip\no 
\textit{Claim: }$S_a^2(J_a^{1}) \subset J_a^{1}.$ 

\medskip\no Suppose $z \in U_1 \cap J_a$. If $S_a(z) \in U_0 \times U_c \times U_c$ or  $S_a(z) \in U_0 \times U_c \times U_\infty$ then $z$ should lie in the Fatou set of $S_a$. Hence $S_a(z) \in U_0 \times U_c \times U_0 \subset U_2$ and $S_a^2(z) \in U_1.$ Now $z \in J_a^1$ means there exists $n_z \ge 0$ such that $S_a^{2n}(z) \in U_1 \cap J_a$ for all $n \ge n_z.$ Thus the claim.

\medskip\no Also $S_a^{-2}(z) \in J_a^{1}$, i.e., $S_a^{-2}(J_a^{1}) \subset J_a^{1}.$ Thus $S_a^2(J_a^{1})=J_a^{1}=S_a^{-2}(J_a^{1}).$ A similar argument gives $S_a^2(J_a^{2})=J_a^{2}=S_a^{-2}(J_a^{2}).$ 
\end{proof}
\begin{lem}\label{lemma 3}
For a fixed $z \in J_a^{1} \cup J_a^{2}$ define the following set:
\[ J_{a,z}=\big\{w: w \text{ is a limit point of the sequence } \{S_a^{2n}(z)\}\big\}.\]
Define 
\[ J_1=\ov{\bigcup_{z \in J_a^{1}} J_{a,z}},\; J_2=\ov{\bigcup_{z \in J_a^{2}} J_{a,z}}\text{ and } J_3=\ov{J_a^3}.\]
Then $J_1$, $J_2$ and $J_3$ are closed compact subsets contained in $U_1$, $U_2$ and $U_3$ respectively and they are completely invariant under $S_a^2.$
\end{lem}
\begin{proof}
Let 
\[ J^{1}={\bigcup_{z \in J_a^{1}} J_{a,z}},\text{ and } J^{2}={\bigcup_{z \in J_a^{2}} J_{a,z}}.\]
For $w \in J^{1}$ there exists $z \in J_a^{1}$ and a subsequence of natural numbers $\{n_k\}$ such that 
$S_a^{2n_k}(z) \to w.$ But $S_a^{2n_k}(z) \in U_1$ for $k \ge k_0$, i.e., $w \in U_1$ as $U_1$ is a closed set. Hence  $J^1 \subset U_1.$  A similar argument gives $J^2 \subset U_2.$

\medskip\no 
\textit{Claim: } $J^{1}$ and $J^{2}$ are completely invariant under $S_a^2.$

\medskip\no 
For $w \in J^{1}$ there exists $z \in J_a^{1}$ such that $w \in J_{a,z}$, i.e., $S_a^2(w) \in J_{a,S_a^2(z)}$. By Lemma \ref{lemma 2}, $S_a^2(z) \in J_a^1$, i.e., $S_a^2(w) \in J^1.$ Hence $S_a^2(J^{1}) \subset J^{1}.$ Now $S_a^{2n_k}(z) \to w$ where $n_k \ge 1$ for $k > 1$, i.e., $S_a^{-2}(w) \in J_{a,S_a^{-2}(z)}$. Again by Lemma \ref{lemma 2}, $S_a^{-2}(J^{1}) \subset J^{1}.$ A similar argument gives that $J^{2}$ is also completely invariant under $S_a^2.$

\medskip\no Since $J_1=\ov{J^1}$, $J_2=\ov{J^2}$, $J_3=\ov{J_a^3}$ and $S_a$ is one--one for $a \neq 0$, the result follows.
\end{proof}
%
\no Let $W_c$ be a relatively compact open subset of the bounded Fatou--components of $p$ containing $U_c$ and $R'>R$. Note that $U_2  \subset D_{R'} \times W_c \times V$, $U_1  \subset D_{R'} \times V \times W_c$ and $U_3 \subset D_{R'} \times V \times V$ where $V$ be the neighbourhood of $J_p$ containing $U_0$ (that admits the hyperbolic metric $\rho$ such that $p$ is expanding with respect to $\rho$ in $V$). 

\medskip\no When $a=0$, i.e., $S_0^2(z_1,z_2,z_3)=(z_3,p(z_2),p(z_3))$ and 
\[ DS_0^2(z_1,z_2,z_3)= \begin{pmatrix}
0 &0 &1 \\ 0 &p'(z_2) &0 \\ 0 &0 &p'(z_3)
\end{pmatrix}.\] Let the eigenvalues of $DS_0^2(z)$ be denoted by $\la_i^z(0)=0, p'(z_2) \text{ and }p'(z_3)$ for $i=1,2,3$ and
\[ E_i^z(0)=\text{The eigenspace corresponding to the eigenvalue }\la_i^z(0) .\]
Then the eigenspaces $E_i^z(0)$ are as follows:
\begin{align*}
&E_1^z(0)=\{(t,0,0): t \in \mbb C\}, \\
&E_2^z(0)=\{(0,t,0): t \in \mbb C\}, \\
&E_3^z(0)=\{(t,0,tp'(z_3)): t \in \mbb C\}.
\end{align*}
 Here by $E_i^z(a)$, we will denote the eigenspaces corresponding to the eigenvalue $\la_i^z(a)$ of $DS_a^2(z)$ for $a \neq 0.$ 

\medskip\no 
\begin{lem}\label{invariance of subspaces}
 For every $1 \le i \le 3$ and $z \in \mbb C^3$, $E_i^z(0)$ is invariant under $DS_0^2(z)$.
\end{lem}
\medskip\no 
\begin{proof}
Note that $DS_0^2(z)(E_1^z(0))=\{0\} \in E_1^{S_0^2(z)}(0)$ and $$DS_0^2(z)(0,t,0)_{S_0^2(z)}=(0,tp'(p(z_2)),0) \in E_2^{S_0^2(z)}(0).$$
Also $$DS_0^2(z)(t,0,tp'(z_3))_{S_0^2(z)}=\big(tp'(p(z_3)),0,tp'(p(z_3))^2\big)=\big(\tilde{t},0,\tilde{t}p'(p(z_3))\big) \in E_3^{S_0^2(z)}(0).$$
Hence the proof.
\end{proof}
\begin{prop}\label{hyperbolic_prop}
There exists $A>0$ such that for $0<|a|<A$, 
\begin{itemize}
\item[(i)] $D_{R'} \times V \times W_c$ admits a Riemannian metric (equivalent to the Euclidean metric) on such that $S_a^2$ is hyperbolic on $J_1.$
\item[(ii)] $D_{R'} \times W_c \times V$ admits a Riemannian metric (equivalent to the Euclidean metric) on such that $S_a^2$ is hyperbolic on $J_2.$
\item[(iii)] $D_{R'} \times V \times V$ admits a Riemannian metric (equivalent to the Euclidean metric) on such that $S_a^2$ is hyperbolic on $J_3.$
\end{itemize} 
\end{prop}
\begin{proof}
We will prove statement (ii) first.

\medskip\no For $v=(0,t,0) \in E_2^z(0)$, consider the identification, $\phi_2(v)=t \in T_{z_2}W_c. $ Define (with abuse of notation) the $$\|v\|^z_{\mbb H}=\|\phi_2(v)\|^{z_2}_{\mbb H}$$ where $\mbb H$ is the hyperbolic metric in the components of interior of $K_p$ (i,e., the bounded Julia set of $p$).  Now with respect to the, hyperbolic metric there exists a constant $0\le C<1$
such that for $v \in E_2^z(0)$
\begin{align}\label{contarctingequation}
\|DS_0^2(z)v\|^{S_0^2(z)}_{\mbb H}=\|p'(z_2)\phi_2(v)\|^{p(z_2)}_{\mbb H}< C\|\phi_2(v)\|^{z_2}_{\mbb H}=C\|v\|^z_{\mbb H}.
\end{align}

\no For $v=(t,0,tp'(z_3)) \in E_3^z(0)$, consider the identification, $\phi_3(v)=t \in T_{z_3}V. $ Now $V$ admits a hyperbolic metric $\rho$ such that $p$ is expanding. Define $\|v\|_{\eta}=\|\phi_3(v)\|_{\rho}$, i.e., there exists a $\la>1$ such that
\begin{align}\label{expandingequation}
\|DS_0^2(z)v\|^{S_0^2(z)}_{\eta}=\|p'(z_3)\phi_3(v)\|^{p(z_3)}_{\rho}> \la \|\phi_3(v)\|_{\eta}^{z_3}=\la \|v\|^{z}_{\eta}. 
\end{align}

\no For $z \in D_{R'} \times W_c \times V$, and $v=(v_1,v_2,v_3)\in E_1^z(0) \oplus E_2^z(0) \oplus E_3^z(0) = T_z (D_{R'} \times W_c \times V)$ we define the metric as:
\[ \|v\|_{\varrho_0}^z=\|v_1\|^z_{\mbb E}+\|v_2\|_{\mbb H}^z + \|v_3\|_{\eta}^{z}.\]
Let $E_u^z(a)= E_3^z(a)$ and $E_s^z(a)=E_1^z(a) \oplus E_2^z(a).$ From invariance of $E_s^z(0)$ and $E_u^z(0)$ and from (\ref{contarctingequation})and (\ref{expandingequation}) it follows that there exists $\la_0>1$ such that:
\begin{align*}
& \|DS_0^{2}(z)v^u\|_{\varrho_0}^{S_0^{2}(z)}> \la_0 \|v^u\|_{\varrho_0}^z\\
& \|DS_0^{2}(z)v^s\|_{\varrho_0}^{S_0^{2}(z)}< \la_0^{-1} \|v^s\|_{\varrho_0}^z .
\end{align*}
Now using similar arguments as in the proof of Theorem \ref{main theorem 1} in $\mbb C^3$, we define the metric $\varrho_a$ in $ D_{R'} \times W_c \times V.$ Modify the choice of $|a|$, so that in appropriate cones $C^s_z(\rho_1,a)$ and $C^u_z(\rho_1,a)$ satisfies
\begin{itemize}
\item[(i)] $DS_a^{2}(z)(C^s_z(\rho_1,a)) \subset \text{int}\; C^s_{S_a^{2}(z)}(\rho_1,a) $ and $DS_a^{-2}(z)\big(C^u_{S_a^{2}(z)}(\rho_1,a)\big) \subset \text{int}\; C^u_{z}(\rho_1,a).$

\medskip 
\item[(ii)] There exists $\la_1>1$ such that for every $z \in J_2 $
\[ \|DS_a^{-2}(z)(v)\|_{\varrho_a}^{S_a^{-2}(z)} \ge \la_1 \|v\|_{\varrho_a}^z \text{ for } v\in C^s_{S_a^2(z)}(\rho_1,a)\]
and
\[\|DS_a^{2}(z)(v)\|_{\varrho_a}^{S_a^{2}(z)} \ge \la_1 \|v\|_{\varrho_a}^z \text{ for } v\in C^u_z(\rho_1,a).\]
\end{itemize}
Finally appeal to Corollary 6.4.8 in \cite{KatokBook}, to conclude $S_a^2$ is hyperbolic on $J_2.$ Also with similar argument as in proof Theorem \ref{main theorem 1} for $\mbb C^3$ it follows that the metric $\varrho_a$ is equivalent to the Euclidean metric.

\medskip\no For statement (i) and (iii) the stable and unstable directions will change. 

\medskip\no To prove (i) we consider $E^z_s(a)=E^z_1(a) \oplus E_z^3(a)$ and $E^z_u(a)=E^z_2(a)$ The $\varrho_0$ metric on the tangent space, i.e., for $a=0$ and $v=(v_1,v_2,v_3) \in E_1^z(0) \oplus E_2^z(0) \oplus E_3^z(0)= T_z(D_{R'} \times V \times W_c)$ is defined as:
\[ \|v\|^z_{\varrho_0}=\|v_1\|^{z}_{\mbb E}+\|v_2\|^{z}_{\eta}+\|v_3\|^z_{\mbb H}.\]

\medskip\no To prove (iii) we consider $E^z_s(a)=E^z_1(a)$ and $E^z_u(a)=E^z_2(a) \oplus E_3^z(a).$ The $\varrho_0$ metric on the tangent space, i.e., for $a=0$ and $v=(v_1,v_2,v_3) \in E_1^z(0) \oplus E_2^z(0) \oplus E_3^z(0)= T_z(D_{R'} \times V \times V)$ is defined as:
\[ \|v\|^z_{\varrho_0}=\|v_1\|^{z}_{\mbb E}+\|v_2\|^{z}_{\eta}+\|v_3\|^z_{\eta}.\]

\medskip\no Now by arguing similarly as in the proof of statement (ii), the proof is complete.
\end{proof}

\no Finally, we proof Theorem \ref{main theorem 2}.
\begin{proof}[Proof of Theorem \ref{main theorem 2} for $k=3$]
For $\nu=1$, $S_a^{2}$ is a regular polynomial automorphism of $\mbb C^3$ by Proposition \ref{shift-regular} and $S_a$ is hyperbolic on $J_a$ for sufficiently small $|a|$ by Theorem \ref{main theorem 1}. Since $U^+_a$ is a connected component of the Fatou set, from Theorem 4.2 of \cite{safikov-wolf}, it follows that $S_a^{2}$ satisfies all the properties (a)--(c). But Fatou set of $S_a^{2}$ is same as $S_a$. 

\medskip\no 
For $\nu=2$, suppose there exists a wandering domain of $S_a$ where $0<|a|<A$, $A$ as obtained in Proposition \ref{hyperbolic_prop}. Let $C$ be a wandering domain of $int(K_a^+)$ and $z \in C.$ Define the set $L_z$ as follows:
\[ L_z=\{w: w\text{ is a limit point of } \{S_a^n(z)\}\}.\]
Recall that $V_R$ denote the polydisc of radius $R$ at the origin, where $R$ is the radius of filtration for $S_a.$ Since $S_a^n(z) \in V_R$ for sufficiently large $n$, $L_z \subset V_R$ and $L_z$ is non-empty. For any $w \in L_z$, it is easy to see that $S_a(w)$ and $S_a^{-1}(w) \in L_z$, i.e., $S_a(L_z)=L_z=S_a^{-1}(L_z).$ Thus $L_z \subset K_a.$ But from Lemma \ref{k-}, it follows that $L_z \subset K_a^+$ and $L_z \subset J_a^-.$

\medskip\no 
\textit{Claim: } $L_z$ is disjoint from $int(K_a^+).$

\medskip\no 
Suppose $w \in L_z \cap int(K_a^+)$ and let $C_0$ be the component of $int(K_a^+)$ that contains $w.$ For $n$, sufficiently large there exists distinct $n_1$ and $n_2$ such that $S_a^{n_1}(z)$ and $S_a^{n_2}(z) \in C_0$, i.e., $S_a^{n_1}(C_0)=S_a^{n_2}(C_0).$ This is a contradiction to the fact that $z$ lies in a wandering Fatou component of $S_a.$  

\medskip\no 
Hence $L_z \in J_a$ and $z \in W^s(J_a)$ where
$$W^s(J_a)=\{z \in \mbb C^3: \text{dist}(S_a^{n}(z),J_a) \to 0 \text{ as } n \to \infty\}.$$ Now $W^s_{S_a^2}(J_1 \cup J_2 \cup J_3) \subset W^s(J_a)$ where
\[W^s_{S_a^2}(J_1 \cup J_2 \cup J_3)=\{z \in \mbb C^3: \text{dist}(S_a^{2n}(z),J_1 \cup J_2\cup J_3) \to 0 \text{ as }n \to \infty\}\] and $J^1$, $J^2$ and $J^3$ is as obtained in Lemma \ref{lemma 2} and \ref{lemma 3}.

\medskip\no 
\begin{lem}
$W^s(J_a)=W^s_{S_a^2}(J_1 \cup J_2 \cup J_3) \subset J_a^+$.
\end{lem}
\begin{proof} 
Suppose $z \in W^s(J_a)=W^s_{S_a^2}(J_a).$ Then there exists $w_n \in J_a$ such that $\text{dist}(S_a^{2n}(z),w_n) \to 0$ as $n \to \infty.$ Let $A(\{w_n\})$ the set of limits of the sequence $\{w_n\}$, i.e.,
\[A(\{w_n\})=\{w: w_{n_k} \to w\}.\]
Since $J_a$ is closed, $A(\{w_n\}) \subset J_a.$ Let $w_0 \in A(\{w_n\})$, then $w_0 \in J_a^i$ for some $i=1,2,\text{ or }3$, where $J_a^i$ is as defined in Lemma \ref{lemma 2}.

\medskip\no 
\textit{Case 1: }If $w_0 \in J_a^3$, then $\text{dist}(S_a^{2n_k}(z),w_0) \to 0$, i.e., $z \in W^s_{S_a^2}(J_3).$

\medskip\no 
\textit{Case 2: }If $w_0 \in J_a^1$ then consider $J_{a,w_0}$ as in Lemma \ref{lemma 3}. Let $\tilde{w} \in J_{a,w_0} \subset J_1$, i.e., there exists a subsequence of natural numbers $\{m_l\}$ such that $S_a^{2m_l}(w_0) \to \tilde{w}$ as $ l \to \infty.$ Let $\ep>0$ be arbitrary, then there exists $l_0 \ge 1$ sufficiently large such that $\text{dist}(S_a^{2m_{l_0}}(w_0),\tilde{w})< \ep.$ Now $S_a^{2m_{l_0}}$ is uniformly continuous on $V_R \supset J_a$, i.e., there exists $\delta>0$ such that 
\[ \text{dist}(S_a^{2m_{l_0}}(z),S_a^{2m_{l_0}}(w))< \ep \text{ whenever } \text{dist}(z,w)< \delta\]
for $z,w \in V_R.$ Now buy assumption there exists $k_0 \ge 1$ such that 
\[ \text{dist}(S_a^{2n_k}(z),w_0)< \delta \text{ for } k \ge k_0\]
i.e.,
\[ \text{dist}(S_a^{2n_k+2m_{l_0}}(z),\tilde{w})< \ep \text{ for } k \ge k_0.\]
\textit{Case 3: } Similar computation for $w_0 \in J_a^2.$

\medskip\no Note that $A(\{w_n\})$ is a finite set, i.e., $$\bigcup_{w \in A(\{w_n\})} J_{a,w} \text{ is also finite.}$$ Hence all possible subsequences of $\{S_a^n(z)\}$ converges either in $J_1$, $J_2$ or $J_3$ and thus the proof.
\end{proof}
%
\no Thus $z \in W^s_{S_a^2}(J_1 \cup J_2 \cup J_3).$ By Remark \ref{open neighbourhood}, $S_a^{2n_0}(z)$ either lies in $D_R \times W \times U \cup D_R\times U \times W $ for large enough $n_0.$ 

\medskip\no Suppose $S_a^{2n_0}(z) \in U_1$. Since $L_z \subset J_a$, it follows that $S_a^{2n_0+2n}(z) \in D_{R'}\times W_c \times V$ for every $n \ge 1.$ By the proof of Proposition \ref{hyperbolic_prop}, for $p \in D_{R'}\times W_c \times V$ there exists cone $C_p \subset T_p$ such that for $v \in C_p$
\begin{align*}
\|DS_a^{2}(p)v\|^{S_a^{2}(p)}_{\varrho_a} \ge \la \|v\|^p_{\varrho_a}
\end{align*}
and $DS_a^{2}(p)v \in C_{S_a^2(p)}$ if $S_a^2(p) \in D_{R'}\times W_c \times V.$
Thus
\begin{align}\label{final_1}
 \|DS_a^{2n+2n_0}(z)v\|^{S_a^{2n+2n_0}(z)}_{\varrho_a} \ge \la^n \|v\|^{S_a^{2n_0}(z)}_{\varrho_a}.
\end{align} 
By assumption $S_a^{2n}(z)$ converges uniformly around a neighbourhood of $z$ in the Euclidean norm which means that $S_a^{2n}(z)$ converges uniformly in the metric $\varrho_a$ as well, since the topology induced $\varrho_a$ is equivalent to the topology induced Euclidean norm. But this is not possible from \ref{final_1}. Hence, $z \in J_a^+.$  

\medskip\no A similar argument works if $S_a^{2n_0}(z) \in U_2$ or $S_a^{2n_0}(z) \in U_3$ such that 
\[ S_a^{2n+2n_0} \in D_{R'} \times V \times W_c \text{ or } S_a^{2n+2n_0} \in D_{R'} \times V \times V.\] This proves that $z \in J_a^+$ and hence there does not exists any wandering domain of $S_a.$

\medskip\no 
\textit{Claim: } $J_a$ is the maximal invariant set in  $\tilde{U}=D_R \times W \times U \cup D_R \times U \times W.$

\medskip\no Suppose $\cal{J}$ be an invariant subset of $\tilde{U}$, and $z \in \cal{J}.$ Then for sufficiently large $n_0$, $S_a^{2n+n_0}(z) \in U_i $ for some $1 \le i \le 3.$ Now a similar argument as above, gives that $z \in J_a^+.$ Since $\cal{J}$ is completely invariant $z \in K_a^-=J_a^-$, i.e., $z \in J_a.$ Thus $\cal{J} \subset J_a$ and hence the claim.

\medskip\no Now the proof is similar to the Proof of Theorem 5.6 in \citep{BS1}. However we will revisit the arguments for the sake of completeness.

\medskip\no First we prove that every Fatou component in $K_a^+$ is a basin of attraction of a periodic point. Let $C$ be a Fatou component in $K_a^+$ with period $m$, i.e., $S_a^m(C)=C.$ Consider $C'=C \cap V_R$, $R$ sufficiently large such that $C'$ is a bounded domain in $\mbb C^k$. Now $C \subset V_R \cup V_R^-$ and $S_a^m(V_R) \subset V_R \cup V_R^+.$ Hence $S_a^m(C') \subset C'.$ With abuse of notation, let us assume $m=1$. Since the sequence $\{S_a^n\}$ is normal in $C'$ and the subsequences are uniformly convergent by Theorem 1.1 in \cite{Be}, it follows that either iterates of the points diverge to the boundary or converge towards a submanifold in $C'.$

\medskip\no Note that by filtration properties, every point of $C$ eventually lands in $C'.$ If iterates of point in $C'$ diverges to the boundary, it means every point of $C$ diverges to the boundary of $C$ and $\partial C \subset J_a^+.$ Thus if $L$ is the set of limit points of $S_a$ in $C'$, then $L$ is actually the set of limit point for the iterates of $S_a$ in $C$. Hence Theorem 1.1 in \cite{Be} assures that $L \subset J_a^+.$ Also, as $S_a^{-1}(C)=C$ it follows that $S_a^{-1}(L)=L$. Thus $L \subset J_a$ and $C \subset W^s(J_a) \subset J_a^+$, which is not possible.

\medskip\no So there  exists a complex connected submanifold $M \in C' \subset C$ such that $\{S_a^n\}$ converge to $M$ and $S_a(M)=M.$ If dimension $M \ge 1$ then $M$ cannot be compact in $C$, i.e., $\partial M \subset \partial C \subset J_a.$ Further Theorem 1.1 from \cite{Be}, says that ${S_a}_{|M} \in {\sf Aut}(M)$, i.e., $S_a$ is an isometry on $M$ with respect to the Kobayashi metric on $C'$. 

\medskip\no Let $\tilde{M}=M \setminus \tilde{U} \subset C'$, where $\tilde{U}$ is the open set in which $J_a$ is maximal. Clearly $\tilde{M}$ is compact in $C'.$ Now since $S_a$ is an isometry and orbit of any point in the interior of $C'$ does not diverge to the boundary, the set
\[Q=\ov{ \bigcup_{n=0}^\infty S_a^{n}(\tilde{M})}\]
is a compact set in $C'.$ Thus there exists $p \in M \setminus Q \subset C'$ such that $S_a^n(p) \in \tilde{U}$ for every $n \ge 0.$ Since $p \notin J_a$ and $J_a$ is the maximal invariant set in $\tilde{U}$ this is not possible. Hence the dimension of $M$ is zero, i.e., $M$ is a single point. This proves (b).

\medskip\no To prove (c), suppose there are infinitely many attracting periodic points, say $\{p_i\}.$ Then $\{p_i\} \in V_R$ (where $R$ is the radius of filtration for $|a|<1$). Consider $L$ to be the set of limit points of the sequence $\{p_i\}.$ By repeating similar argument as in the proof of part (a), one can prove that $L \subset J_a.$ This means there exists $p_{i_0}$ such that the orbit of $p_{i_0} \in \tilde{U}.$ But this contradicts the fact that $J_a$ is the maximal invariant set in $\tilde{U}.$ Thus the proof.
\end{proof}

\begin{ex} Let $p(z)=z^2$ and $$S_a^2(z_1,z_2,z_3)=(z_3,az_1+z_2^2,az_2+z_3^2) \text{ and }S_0(z_1,z_2,z_3)=(z_2,z_3,z_2^2).$$
Let $\Gamma_2=(z_1,z_2,z_1^2)$, i.e., graph of the function $\psi: \mbb C^2 \to \mbb C$, $\psi(z_1,z_2)=z_1^2.$ Hence the map $\phi: \mbb C^2 \to \Gamma_2$ defined as $\phi(z_1,z_2)=(z_1,z_2,\psi(z_1,z_2))$ is a biholomorphism. Also $$S_0^2(\Gamma_2)=\Gamma_2, \text{ i.e., } \phi \circ p_2 \circ \phi^{-1}(\Gamma_2)=\Gamma_2$$
where $p_2(z_1,z_2)=(p(z_1),p(z_2))=(z_1^2,z_2^2).$ Note that the Julia set for $p_2$ in $\mbb C^2$ is $$J_{p_2}=\bar{\mbb D} \times S^1 \cup S^1 \times \bar{\mbb D}$$ where $\mbb D$ is the unit disc and $S^1$ is the unit circle in $\mbb C.$ Hence $J_0=\phi(J_{p_2}).$ Now from definition of the subset $J_i^0$, $1 \le i \le 3$ as in Lemma \ref{lemma 3}, they are as follows:
\[ J_1^0=\phi(\{0\} \times S^1), \;\; J_2^0=\phi(S^1\times \{0\}) \text{ and } J_3^0=\phi(S^1 \times S^1).\]
\end{ex}
\section{Sketch of the proof of Theorem \ref{main theorem 1} and \ref{main theorem 2}}
\no For any $1 \le \nu \le k-1$, recall that the map $S_a^{\nu}$ is of the form
\[ S_a^{\nu}(z_1,\hdots,z_k)=(z_{\nu+1},\hdots,z_k, az_1+p(z_{k-\nu+1}),\hdots,az_{\nu}+p(z_k)).\]
Let $D_R $ be the open disc of radius $R$ in $\mbb C$, where $R$ is the filtration radius for $S_a$, $0<|a|<A.$ Recall from the previous section 
$\mbb C=U_0 \cup U_c \cup U_{\infty}$ where we can consider $U_0=\bar{U}$, $U_c$ and $U_{\infty}$ to be closed subsets of $\mbb C.$ Let $U_{i_1,\hdots,i_\nu}$ denote the following sets for $i_j=0,c \text{ or } \infty$ and $1 \le j \le \nu$
\[ U_{i_1,\hdots,i_\nu}=\ov{D_R}^{k-\nu}\times U_{i_1}\times \hdots U_{i_\nu}.\]
\begin{lem}\label{structure of J}
If $|a|$ is sufficiently small then $$J_a \subset \bigcup_{i_j=0,c} U_{i_1,\hdots,i_\nu} \setminus U_{c,c,\hdots,c}$$
\end{lem}
\begin{proof}
We will prove this by contradiction. Let $z \in J_a$ such that $z_{i_0} \in U_{\infty}$ for some $k-\nu+1 \le i_0 \le k. $ Then $\pi_{i_0} S_a^{n\nu}(z) \to \infty$ as $n \to \infty$, i.e., $z$ lies in the Fatou set of $S_a$, which is not possible. Thus we prove that 
$$J_a \subset \bigcup_{i_j=0,c} U_{i_1,\hdots,i_\nu}.$$
Suppose $\pi_i(z) \in U_c$ for every $k-\nu+1 \le i \le k $, then there exists a neighbourhood of $z$ such that $S_a^n$ is bounded. Hence $z$ lies in Fatou set of $S_a.$ Thus the proof.
\end{proof}
\no For every $U_{i_1,\hdots,i_\nu}$ let 
$$\cal{N}_1^{i_1,\hdots,i_\nu}=\{j: i_j=c, 1 \le j \le \nu\} \text{ and } \cal{N}_2^{i_1,\hdots,i_\nu}=\{j: i_j=0, 1 \le j \le \nu\}.$$
Thus from Lemma \ref{structure of J} if $J_a \cap U_{i_1,\hdots,i_\nu} \neq \emptyset$ the $$\# \cal{N}_1^{i_1,\hdots,i_\nu} +\# \cal{N}_2^{i_1,\hdots,i_\nu}=\nu \text{ and }\# \cal{N}_2^{i_1,\hdots,i_\nu} \ge 1.$$
Also let \[ \cal{M}=\{(i_1,\hdots,i_\nu): i_j=0 \text{ or } c \text{ for every } 1 \le j \le \nu\}.\]
\begin{lem}\label{U_0 time}
For $(i_1,\hdots,i_\nu) \in \cal{M}$ let
\[ J_a^{i_1,\hdots,i_\nu}=J_a \cap U_{i_1,\hdots,i_\nu}.\] 
\begin{itemize}
\item[(i)]For $\# \cal{N}_2^{i_1,\hdots,i_\nu}=1$ let 
\[ J_a^{(i_1,\hdots,i_\nu)} = \ov{\bigcup_{n=0}^{\infty} S_a^{-\nu n}(J_a^{i_1 \hdots,i_\nu})}\] 
and \[J_1=\bigcup_{\# \cal{N}_2^{i_1,\hdots,i_\nu}=1} J_a^{(i_1,\hdots,i_{\nu})}.\]
Then $J_a^{(i_1,\hdots,i_\nu)}$ is completely invariant under $S_a^{\nu}.$ 

\medskip\no 
\item[(ii)]For $\# \cal{N}_2^{i_1,\hdots,i_\nu}=m$ such that $2 \le m \le \nu$ let 
\[ J_a^{(i_1,\hdots,i_\nu)} = \ov{\bigcup_{n=0}^{\infty} S_a^{-\nu n}(J_a^{i_1 \hdots,i_\nu})}\setminus \bigcup_{i=1}^{m-1} J_i\] and \[J_m=\bigcup_{\# \cal{N}_2^{i_1,\hdots,i_\nu}=m} J_a^{(i_1,\hdots,i_{\nu})}.\]
Then $J_a^{(i_1,\hdots,i_\nu)}$ is completely invariant under $S_a^{\nu}.$ 
\end{itemize}
\end{lem}
\begin{proof}
The proof is similar to Lemma \ref{lemma 2}. 

\medskip\no
\textit{Case 1: }Let $z \in J_a^{(i_1,\hdots,i_{\nu})}$ such that $\# \cal{N}_2^{i_1,\hdots,i_\nu}=m=1$. Let $\cal{N}_2^{i_1,\hdots,i_\nu}=\{{j_0}\}$. Then $\pi_{k-\nu+j_0}(z) \in U_0$ and $\pi_{k-\nu+j}(z)\in U_c$ for every $1 \le j \neq j_0 \le \nu$, i.e.,$$\pi_{k-\nu+j}S_a^{\nu}(z)=az_j+p(z_{k-\nu+j}) \in U_c \text{ for } j \neq j_0$$ and
$$\pi_{k-\nu+j_0}S_a^{\nu}(z)=az_{j_0}+p(z_{k-\nu+{j_0}}) \in U_c \cup U_0 .$$
But by Lemma \ref{structure of J}, $\pi_{k-\nu+j_0}S_a^{\nu}(z) \in U_0.$ Hence $S_a^{\nu}(J_a^{(i_1,\hdots,i_{\nu})}) \subset J_a^{(i_1,\hdots,i_{\nu})}.$ Also $S_a^{-\nu}(J_a^{(i_1,\hdots,i_{\nu})}) \subset J_a^{(i_1,\hdots,i_{\nu})}$ is true from the definition of $J_a^{(i_1,\hdots,i_{\nu})}.$ Thus the proof. 

\medskip\no 
\textit{Case 2: }Let $z \in J_a^{(i_1,\hdots,i_{\nu})}$ such that $\# \cal{N}_2^{i_1,\hdots,i_\nu}=m \ge 2.$ Assume the statement is true for all $i$, $1 \le i \le m-1.$ Then $\pi_{k-\nu+j}(z) \in U_0$ for $j \in \cal{N}_2^{i_1,\hdots,i_\nu}$ and $\pi_{k-\nu+j}(z)\in U_c$ for every $j \in \cal{N}_1^{i_1,\hdots,i_\nu}.$, i.e.,$$\pi_{k-\nu+j}S_a^{\nu}(z)=az_j+p(z_{k-\nu+j}) \in U_c \text{ for } j \in \cal{N}_1^{i_1,\hdots,i_\nu}$$ and
$$\pi_{k-\nu+j}S_a^{\nu}(z)=az_{j}+p(z_{k-\nu+{j}}) \in U_c \cup U_0 \text{ for } j \in \cal{N}_2^{i_1,\hdots,i_\nu}.$$
But if $\pi_{k-\nu+j}S_a^{\nu}(z) \in U_c$ then $z \in J_{m-1}$, which is not possible. Hence $S_a^{\nu}(J_a^{(i_1,\hdots,i_{\nu})}) \subset J_a^{(i_1,\hdots,i_{\nu})}.$ Also $S_a^{-\nu}(J_a^{(i_1,\hdots,i_{\nu})}) \subset J_a^{(i_1,\hdots,i_{\nu})}$ since each $J_i$, $1 \le i \le m-1$ is completely invariant under $S_a^\nu.$ Thus the proof. 
\end{proof}
%
\no For $z \in J_a$, let  $L(z)$ denote the set of limit points of $\{S_a^{n \nu}(z)\}$, i.e.,
\[ L(z)=\{w: w \text{ is a limit point of } S_a^{n \nu}(z)\}.\] 
Now for $i_j=0 \text{ or }c$ whenever $1 \le j \le \nu$ we define the sets $J_{i_1,i_2,\hdots,i_\nu}$ as follows:
\[ J_{i_1,i_2,\hdots,i_\nu}=\ov{\bigcup_{z \in J_a^{(i_1,\hdots,i_{\nu})}} L(z)}.\]
\begin{lem}
$J_{i_1,\hdots,i_\nu}$ is a closed compact subset of $U_{i_1,\hdots,i_{\nu}}$ which is completely invariant under $S_a^{\nu}.$
\end{lem}
\begin{proof}
The proof is similar to Lemma \ref{lemma 2}.
\end{proof}
\no Let $V_c$ be an open set in $\mbb C$ relatively compact in the Fatou components of $p$ and containing the set $U_c$, $V_0$ a slightly bigger neighbourhood of $U_0$, which admits the hyperbolic metric $\rho.$ Also let $R'>R$ and $D_{R'}$ be the open disc of radius $R'.$
As before we consider the sets
$$V_{i_1,\hdots,i_{\nu}}=D_{R'}^{k-\nu} \times V_{i_1} \times \hdots \times V_{i_{\nu}}$$
where $i_j=0,c$ for $1 \le j \le \nu.$ Then $U_{i_1,\hdots,i_\nu}$ is a compact subset of $ V_{i_1,\hdots,i_\nu}.$

\medskip
\no Note that $J_{i_1,i_2,\hdots,i_{\nu}}$ may be empty. We will ignore such situations, i.e., let
\[ \cal{I}=\{(i_1,\hdots,i_\nu): J_{i_1,i_2,\hdots,i_{\nu}} \neq \emptyset \text{ where } i_j=0 \text{ or } c \text{ for every } 1 \le j \le \nu \}.\]

\no For $a=0$ and $z \in \mbb C^k$ observe that
\begin{align}
DS_0^{\nu(k-\nu)}(z)=\begin{pmatrix}
0_{k-\nu \times k-\nu} & [*]_{\nu \times k-\nu} \\
0_{\nu \times k-\nu} &A_{\nu \times \nu}
\end{pmatrix}
\end{align}
where $A$ is a diagonal matrix given by $A=Diag\big(p^{k-\nu}(z_{k-\nu+1}),\hdots,p^{k-\nu}(z_k)\big),$ i.e.,
\begin{align}
det(DS_0^{\nu(k-\nu}(z)-\la \text{Id})=\la^{k-\nu}\Pi_{i=1}^\nu(p^{k-\nu}(z_{k-\nu+i})-\la).
\end{align}

\medskip\no 
Let $\la_i^z(a)$ denote the eigenvalues of $DS_a^{\nu}(z)$, i.e.,
\[ \la_i^z(0)=0, \text{ for } 1 \le i \le k-\nu \text{ and } \la_i^z(0)=(p^{k-\nu})'(z_i) \text{ for } k-\nu+1 \le i \le k.\]
Also let $$E_0^z(a)=\text{ the eigenspace of the eigenvalues }\la_i^z(a) \text{ where } 1 \le i \le k-\nu$$
and for every $1 \le j \le \nu$ let,
$$E_j^z(a)=\text{ the eigenspace of the eigenvalues }\la_{k-\nu+j}^z(a).$$
\begin{lem}
 $E_j^z(0)$ is an invariant under $DS_0^{\nu(k-\nu)}(z)$ for every $0 \le j \le \nu. $ 
\end{lem}
\begin{proof}
Note that $$E_0^z(0)={\sf Span}\{\underbrace{(0,\hdots,1,\hdots,0)}_{\text{i-th position}}: 1 \le i \le k-\nu\}$$
and $DS_0^{\nu(k-\nu)}(z)(E_0^z(0))=\{0\}$. Hence $E_0^z(0)$ is invariant under $S_0^{\nu(k-\nu)}.$

\medskip\no 
\textit{Case 1:} For $\nu=1$ or $k-\nu=1$ 
\begin{align*}
&S_0^{k-1}(z_1,\hdots,z_k)=(z_k,p(z_k),\hdots,p^{k-1}(z_k)) \\
\text{ or } &S_0^{k-1}(z_1,\hdots,z_k)=(z_k,p(z_2),\hdots,p(z_k))
\end{align*}
respectively.

\medskip\no 
\textit{Case 2:} For $2 \le \nu \le k-2$ and $k \ge 4$, note that $\nu(k-\nu) \ge k.$ Let $m \ge 1$ be the largest integer such that $k=m \nu+r$ where $0 \le r < \nu.$  Then for $1 \le i \le r$
\[ \pi_i \circ S_0^{\nu(k-\nu)}(z)=p^{k-\nu-m}(z_{k-r+i})\]
and $r+(j-1) \nu+1 \le i \le j \nu+r$ for $1\le j \le m$
\[ \pi_i \circ S_0^{\nu(k-\nu)}(z)=p^{k-\nu-m+j}(z_{k-\nu+[i-r]})\]
where $[i-r]=(i-r)\mod \nu.$

\medskip\no Thus each coordinate of $S_0^{\nu(k-\nu)}$ is a function of the last $\nu-$coordinates. Observe that, for $\nu-r+1 \le i \le \nu$, $E_i^z(0)={\sf span}\{v\}$ where
\[v=\begin{cases}
\pi_{i-\nu+r+j\nu}(v)=(p^{k-\nu-m+j})'(z_{k-\nu+i}) \text{ for } 0 \le j \le m\\ \pi_l(v)=0 \text{ otherwise }
\end{cases}\] 
and for $1 \le i \le \nu-r$, $E_i^z(0)={\sf span}\{v\}$ where
\[v=\begin{cases}
\pi_{j\nu+i+r}(v)=(p^{k-\nu-m+1+j})'(z_{k-\nu+i}) \text{ for } 0 \le j \le m\\ \pi_l(v)=0 \text{ otherwise }
\end{cases}.\] 
Hence, $DS_0^{\nu(k-\nu)}(E_i^z(0)) \subset E_i^{S_0^{\nu(k-\nu)}(z)}(0).$
\end{proof}
\begin{prop}
For every $x \in \cal{I}$, $V_x$ admits a Riemannian metric (equivalent to the hyperbolic metric) such that $J_x$ is hyperbolic set for $S_a^{\nu(k-\nu)}$ for sufficiently small choice of $|a|.$
\end{prop}
\begin{proof}
The proof is similar to Proposition \ref{hyperbolic_prop}. We will outline the main steps.

\medskip\no 
\textit{Step 1: }For $1 \le j \le \nu$, there exists an identification $E_j^z(0)$ with the tangent space $z_{k-\nu+j}$, i.e.,  $\phi_j^z(E_j^z(0))=T_{z_{k-\nu+j}} \mbb C$ such that if $v \in E_j^z(0)$
\[ \phi_j^{S_0^\nu(z)}(DS_0^{\nu(k-\nu)}(z)v)= (p^{k-\nu})'(z_{k-\nu+j})\phi_j(v) \in T_{p^{k-\nu}(z_{k-\nu+j})} \mbb C.\]
\textit{Step 2: }Now $z \in V_x$ where $x \in \cal{I}.$ Note that $x$ is a $\nu-$tuple of symbols, i.e., $x=(x_1,x_2,\hdots,x_\nu).$ Let $\cal{J}_1=\{i: x_i=c\}$ and $\cal{J}_2=\{i: x_i=0\}$ i.e., $z_{k-\nu+i} \in V_c$ if $i \in \cal{J}_1$ and $z_{k-\nu+i} \in V_0$ if $i \in \cal{J}_2.$ Recall that in $V_c$ the action of $p$ is strictly contracting with respect to the standard hyperbolic metric (i.e., $\mbb H$) and in $V_0$ the action of $p$ is strictly increasing with respect to the hyperbolic metric $\rho.$ 

\medskip\no Let $v \in T_z V_x=\oplus_{j=0}^{\nu}E_j^z(0)$ for $z \in V_x$, then $v=(v_0,\hdots,v_\nu)$ where $v_j \in E_j^z(0)$ for $0 \le j \le \nu.$ Define the metric on $V_x$ as 
\[ \|v\|_{\varrho_0}^z=\|v_0\|_{\mbb E}^z+ \sum_{i \in \cal{J}_1} \|\phi_j^z(v_j)\|_{\mbb H}^{z_{k-\nu+j}}+\sum_{i \in \cal{J}_2} \|\phi_j^z(v_j)\|_{\rho}^{z_{k-\nu+j}}.\]
\textit{Step 3: }Let $E_s^z(a)=E_0^z(a)\oplus_{j \in \cal{J}_1} E_j^z(a)$ and $E_u^z(a)=\oplus_{j \in \cal{J}_2}E_j^z(a).$ From invariance of $E_s^z(0)$ and $E_u^z(0)$ a it follows that there exists $\la_0>1$ such that:
\begin{align*}
& \|DS_0^{\nu(k-\nu)}(z)v^u\|_{\varrho_0}^{S_0^{\nu(k-\nu)}(z)}> \la_0 \|v^u\|_{\varrho_0}^z\\
& \|DS_0^{\nu(k-\nu)}(z)v^s\|_{\varrho_0}^{S_0^{\nu(k-\nu)}(z)}< \la_0^{-1} \|v^s\|_{\varrho_0}^z .
\end{align*}
\textit{Step 4: }Now as in the proof of Proposition \ref{hyperbolic_prop}, the choice of $a$ can be modified, i.e., there exists $A>0$ such that for $0<|a|<A$, the action of $S_a^{\nu(k-\nu)}$ is hyperbolic on $J_{x}$ for every $x \in \cal{I}.$
\end{proof}
\begin{rem}
Note that $\cal{J}_1$ may be empty, but $\cal{J}_2$ is always non--empty. 
\end{rem}
\begin{rem}
When $\nu=1$, $J_a \subset \ov{D_R}^{k-1}\times U_0 \subset D_{R'}^{k-1} \times V.$ From the invariance of $J_a$ and the Riemannian metric on $D_{R'}^{k-1} \times V$, it follows that the action of $S_a^{k-1}$ is hyperbolic on $J_a.$ Thus Theorem \ref{main theorem 1} is true. 
\end{rem}
\begin{rem}
Observe that
$\mbb C=U_0 \cup U_c \cup U_{\infty}$ where we can consider $U_0=U$, $U_c$ and $U_{\infty}$. If $W_{i_1,\hdots,i_\nu}$ denote the following sets for $i_j=0,c \text{ or } \infty$ and $1 \le j \le \nu$
\[ W_{i_1,\hdots,i_\nu}={D_R}^{k-\nu}\times U_{i_1}\times \hdots U_{i_\nu}.\] Now from the proof of Lemma \ref{structure of J}, if $|a|$ is sufficiently small,
$$J_a \subset \bigcup_{i_j=0,c} W_{i_1,\hdots,i_\nu} \setminus W_{c,c,\hdots,c}.$$
Since $U$ is a open subset of $\mbb C$ the set $$V=\bigcup_{i_j=0,c} W_{i_1,\hdots,i_\nu} \setminus W_{c,c,\hdots,c}$$ is an open subset of $\mbb C^k$ (by similar arguments as in Remark \ref{open neighbourhood}) and $J_a$ is maximal on $V.$
\end{rem}
\no Finally Theorem \ref{main theorem 2} follows by exactly same arguments as in the case $k=3.$
\bibliographystyle{amsplain}
\bibliography{ref}
\end{document}